\documentclass{amsart}

\usepackage{amssymb,latexsym,amsfonts,amsmath}
\usepackage{graphicx}
\usepackage{paralist}
\usepackage{dsfont}
\usepackage{booktabs}
\usepackage{eso-pic}
\usepackage{epsfig}
\usepackage{mathtools}
\usepackage {url}
\usepackage{epstopdf}
\usepackage{tikz}
\usetikzlibrary{arrows,calc,decorations.markings,math,arrows.meta}
\newtheorem{theorem}{Theorem}[section]
\newtheorem{lemma}[theorem]{Lemma}

\newtheorem{proposition}[theorem]{Proposition}
\newtheorem{corollary}[theorem]{Corollary}

\newtheorem{definition}[theorem]{Definition}

\newtheorem{remark}[theorem]{Remark}
\newtheorem{assumption}[theorem]{Assumption}
\numberwithin{equation}{section}

\newcommand{\R}{{\mathbb{R}}}
\newcommand{\F}{{\mathbb{F}}}

\newcommand{\Ze}{{\mathbb Z}}

\newcommand{\Ce}{{\mathcal C}}
\newcommand{\N}{{\mathbb{N}}}

\newcommand{\x}{\mathsf{x}}

\newcommand{\dd}{\mathbf{d}}
\newcommand{\e}{\mathsf{e}}

\DeclareMathOperator{\diff}{d}

\newcommand{\ra}{\rightarrow}

\newcommand{\ul}{\underline}
\newcommand{\ol}{\overline}

\newcommand{\EE}{\mathds{E}}
\newcommand{\PP}{\mathds{P}}

\newcommand{\st}{\mid}
\newcommand{\h}{\hat}

\usepackage{dsfont}

\begin{document}

\begin{abstract}
In this paper, we provide for the first time an automated, correct-by-construction, controller synthesis scheme for a class of infinite dimensional stochastic systems, namely, retarded jump-diffusion systems. First, we construct finite abstractions approximately bisimilar to non-probabilistic retarded systems corresponding to the original systems having some stability property, namely, incremental input-to-state stability. Then, we provide a result on quantifying the distance between output trajectory of the obtained finite abstraction and that of the original retarded jump-diffusion system in a probabilistic setting. Using the proposed result, one can refine the control policy synthesized using finite abstractions to the original systems while providing guarantee on the probability of satisfaction of high-level requirements. Moreover, we provide sufficient conditions for the proposed notion of incremental stability in terms of the existence of incremental Lyapunov functions which reduce to some matrix inequalities for the linear systems. Finally, the effectiveness of the proposed results is illustrated by synthesizing a controller regulating the temperatures in a ten-room building modelled as a delayed jump-diffusion system.
\end{abstract}

\title[Symbolic Models for Retarded Jump-Diffusion Systems]{Symbolic Models for Retarded Jump-Diffusion Systems}

\author[P. Jagtap]{Pushpak Jagtap$^1$} 
\author[M. Zamani]{Majid Zamani$^{2,3}$} 

\address{$^1$Department of Electrical and Computer Engineering, Technical University of Munich, Germany.}
\email{pushpak.jagtap@tum.de}
\address{$^2$Computer Science Department, University of Colorado Boulder, USA.}
\email{majid.zamani@colorado.edu}
\address{$^3$Computer Science Department, Ludwig Maximilian University of Munich, Germany.}
\maketitle

\section{Introduction}
Finite (a.k.a. symbolic) abstraction techniques have gained significant attention in the last few years since they provide tools for automated, correct-by-construction, controller synthesis for several classes of control systems. In particular, such abstractions provide approximate models that are related to concrete systems by aggregating concrete states and inputs to the symbolic ones. Having such finite abstractions, one can make use of the existing automata-theoretic techniques \cite{maler1995synthesis} to synthesize hybrid controllers enforcing rich complex specifications (usually expressed as linear temporal logic formulae or automata on infinite strings) over the original systems. \\
In the past few years, there have been several results providing bisimilar finite abstractions for various continuous-time non-probabilistic as well as stochastic systems. The results include construction of approximately bisimilar abstractions for incrementally stable control systems \cite{pola2008approximately}, switched systems \cite{5342460}, stochastic control systems \cite{zamani2014symbolic}, and randomly switched stochastic systems \cite{zamani2014approximately}. However, the abstractions obtained in these results are based on state-space quantization which suffer severely from the \emph{curse of dimensionality}, i.e., the computational complexity increases exponentially with respect to the state-space dimension of the concrete system.\\
To alleviate this issue, authors in \cite{Girard3} proposed an alternative approach for constructing approximately bisimilar abstractions for incrementally stable non-probabilistic switched systems without discretizing the state-space. The concept is further extended to provide finite abstractions for incrementally stable stochastic switched systems \cite{zamani2015symbolic}, stochastic control systems \cite{zamani2016towards}, and infinite dimensional non-probabilistic control systems \cite{Girard2}. 
For a comparison between state-space discretization based and free approaches, we refer the interested readers to the discussion in Section 5.4 in \cite{zamani2016towards}.\\
On the other hand, retarded stochastic systems are widely used to model various processes in finance, ecology, medical, and engineering (see, e.g., \cite{shaikhet2013lyapunov,bandyopadhyay2008deterministic,6580622}). However, the construction of symbolic models for such classes of systems are still unaddressed due to underlying challenges such as infinite-dimensional functional state-space and dependency on state history. The authors in \cite{pola2010symbolic,pola} provide the construction of abstractions for incrementally stable non-probabilistic time-delayed systems by spline-based approximation of functional spaces. However, the proposed results are complex from the implementation point of view and also suffer from the curse of dimensionality with respect to the state-space dimension of the concrete system. This motivates our work in this paper to provide a scheme for the construction of finite abstractions for a class of infinite dimensional stochastic systems, namely, retarded jump-diffusion systems without discretizing the state-space.\\
The main contribution of this paper is twofold. First, we introduce a notion of incremental input-to-state stability for retarded jump-diffusion systems and provide sufficient conditions for it in terms of the existence of a notion of incremental Lyapunov functions. In the linear case, we show that the sufficient conditions reduce to a matrix inequality. Second, under the assumption of incremental stability property over retarded jump-diffusion systems, we provide a construction of finite abstractions which are approximately bisimilar to the corresponding non-probabilistic version of retarded systems. Then, under some mild assumptions over incremental Lyapunov functions, we obtain a lower bound on the probability such that the distance between output trajectories of the obtained finite abstraction and those of the original retarded jump-diffusion system remains close over a finite time horizon. One can leverage the proposed probability closeness to synthesize a control policy using constructed finite abstractions and refine it back to the original system while providing guarantee on the probability of satisfaction over the original system. Further, we demonstrate the effectiveness of the proposed results by synthesizing a controller keeping temperatures in a comfort zone in a ten-room building modelled as a linear delayed jump-diffusion system. 

\section{Retarded Jump-Diffusion Systems}\label{II}
\subsection{Notations}\label{II1}
Let the triplet $(\Omega, \mathcal{F}, \mathds{P})$ denote a probability space with a sample space $ \Omega $, filtration $ \mathcal{F} $, and the probability measure $ \mathds{P} $. The filtration $\mathds{F}= (\mathcal{F}_s)_{s\geq 0}$ satisfies the usual conditions of right continuity and completeness \cite{Oks_jump}. The symbols $ \N $, $ \N_0 $, $\Ze$, $ \R$, $\R^+,$ and $\R_0^+ $ denote the set of natural, nonnegative integer, integer, real, positive, and nonnegative real numbers, respectively. We use $ \R^{n\times m} $ to denote a vector space of real matrices with $ n $ rows and $ m $ columns. Symbol $ e_i\in \R^n $ denotes a vector whose all elements are zero, except the $ i^{th} $ one, which is one.
Let the family of continuous functions $ \zeta : [-\tau, 0]\rightarrow \R^n$ be denoted by $ \mathcal{C}([-\tau,0];\R^n) $ with a norm $ \|\zeta\|_{[-\tau,0]}:=\sup_{-\tau\leq \theta\leq 0}\|\zeta(\theta)\| $, where $ \|\cdot\| $ denotes the Euclidean norm in $ \R^n $. We denote by $\zeta\equiv c$ a constant function $\zeta\in \mathcal{C}([-\tau,0];\R^n) $ with value $c\in\R^n$. For $ q>0 $ and $ t\in\R_0^+ $, $ \mathbf{L}_{\mathcal{F}_t}^q([-\tau,0];\R^n) $ denotes the family of all $ \mathcal{F}_t $-measurable $ \mathcal{C}([-\tau,0];\R^n) $-valued random processes $ \phi:=\{\phi(\theta)\mid -\tau\leq\theta\leq 0\}$ such that $ \sup_{-\tau\leq \theta\leq 0}\EE[\|\phi(\theta)\|^q]<\infty $. The notation $\mathcal{C}^b_{\mathcal{F}_0}([-\tau, 0];\R^n)$ denotes the family of all bounded $ \mathcal{F}_0 $-measurable $ \mathcal{C}([-\tau,0];\R^n) $-valued random variables. For a matrix $ A \in \R^{n\times m}$,
$ \|A\| $ represents the Euclidean norm of $ A $.  
We use $ \lambda_{\min}(A) $ and $ \lambda_{\max}(A) $ to denote the minimum and maximum eigenvalue of a symmetric matrix $ A $, respectively. The diagonal set $ \bigtriangleup \subset \R^{2n} $ is defined as $ \bigtriangleup=\{(x,x)|x\in \R^n\} $.\\
The closed ball centered at $ x\in\R^m $ with radius $ R $ is defined by $ \mathcal{B}_R(x)=\{y\in\R^m\mid \|x-y\|\leq R \} $. A set $ B\subseteq\R^m $ is called a \textit{box} if $ B=\prod_{i=1}^m[c_i,d_i] $, where $ c_i,d_i\in \R $ with $ c_i<d_i $ for each $ i\in\{1,\ldots ,m\} $. The \textit{span} of a box $ B $ is defined as $ span(B)=\min\{|d_i-c_i|\st i=1,\ldots ,m\} $, where $ |a| $ represents the absolute value of $ a\in\R  $. By defining $ [\R^m]_\eta:=\{z\in\R^m\mid z_i=\frac{k_i\eta}{\sqrt{m}},k_i\in\Ze\} $, the set $ \bigcup_{b\in[\R^m]_\eta}\mathcal{B}_R(b) $ is a countable covering of $ \R^m $ for any $ \eta\in\R^+ $ and $ R\geq \eta/2 $. For a box $ B\subseteq \R^m $ and $ \eta\leq span(B) $, define the $ \eta $-approximation $ [B]_\eta:=[\R^m]_\eta\cap B $. Note that $ [B]_\eta\neq\varnothing $ for any $ \eta\leq span(B) $. We extend the notions of $ span $ and of approximation to finite unions of boxes as follows. Let $ A=\bigcup_{j=1}^MA_j $, where each $ A_j $ is a box. Define $ span(A)=\min\{span(A_j)\mid j=1,\ldots ,M \} $, and for any $ \eta\leq span(A) $, define $ [A]_\eta=\bigcup_ {j=1}^M[A_j]_\eta $.\\
A continuous function $ \gamma  :\R^+_0 \rightarrow \R^+_0 $ belongs to class $ \mathcal{K} $ if it is strictly increasing and $ \gamma(0)=0 $; it belongs to class $ \mathcal{K}_\infty $ if $ \gamma \in \mathcal{K}  $ and $ \gamma(r) \rightarrow \infty $ as $ r \rightarrow \infty $. A continuous function $ \beta:\R^+_0 \times \R^+_0 \rightarrow \R^+_0 $ belongs to class $\mathcal{KL}$ if for each fixed $ s $, the map $ \beta(r,s) $ belongs to class $ \mathcal{K} $ with respect to $ r $ and, for each fixed $ r\neq0 $, the map $ \beta(r,s) $ is decreasing with respect to $ s $ and $ \beta(r,s) \rightarrow 0 $ as $ s\rightarrow \infty $. Given a measurable function $f: \R^+_0 \rightarrow \R^n $, the (essential) supremum of $ f $ is denoted by $ \|f\|_\infty $; we recall that $ \|f\|_\infty$ := (ess)sup$\{\|f(t)\|,t\geq0\}$. 
We identify a relation $ \mathcal{R}\subseteq A\times B $ with the map $ \mathcal{R}:A\ra 2^B $ defined by $ b\in\mathcal{R}(a) $ if and only if $ (a,b)\in\mathcal{R} $.
\subsection{Retarded Jump-Diffusion Systems}
Let $ (W_s)_{s\geq0} $ be an $ \overline{r} $-dimensional $ \mathds{F} $-Brownian motion and $ (P_s)_{s\geq0} $ be an $ \tilde{r} $-dimensional $ \F $-Poisson process. We assume that the Poisson process and the Brownian motion are independent of each other. The Poisson process $ P_s:=[P_s^1;\ldots ; P_s^{\tilde{r}}] $ models $ \tilde{r} $ kinds of events whose occurrences are assumed to be independent of each other.
\begin{definition} \label{definition1}
	A \textit{retarded jump-diffusion system} (RJDS) is a tuple $\textstyle \Sigma_R=(\R^n, \mathcal{X}, \mathsf{U}, \mathcal{U}, f, g, r)$, where:
	\begin{compactitem}
		\item $\R^n$ is the Euclidean space;
		\item $\mathcal{X}$ is a subset of $\Ce([-\tau,0];\R^n)$, for some $\tau\in\R^+_0$;
		\item $ \mathsf{U} \subseteq \R ^m$ is the bounded input set which is finite unions of boxes;
		\item $ \mathcal{U} $ is a subset of the set of all measurable, locally essentially bounded functions of time from $ \R_0^+ $ to $ \mathsf{U} $;
		\item $f : \mathcal{X}\times \mathsf{U} \rightarrow \R^n$, satisfies the following Lipschitz assumption: there exist constants $L_f$, $L_u \in \R^+$, such that
		$\| f(\mathsf{x}_t,u)-f(\hat{\mathsf{x}}_t,\hat{u}) \|\leq L_f \| \x_t-\hat{\x}_t\|_{[-\tau,0]}+L_u \|u-\hat{u}\|$ for all $ \mathsf{x}_t,  \hat{\mathsf{x}}_t\in \mathcal{X} $  and  all $ u, \hat{u}\in\mathsf{U} $;
		\item $g :\mathcal{X} \rightarrow \R^{n\times \overline{r}}$ satisfies the following Lipschitz assumption: there exists a constant $ L_g\in \R^+_0 $ such that $ \| g(\x_t)-g(\hat{\x}_t) \| \leq L_g\| \x_t-\hat{\x}_t \|_{[-\tau,0]}$ for all $ \x_t, \hat{\x}_t\in \mathcal{X}$;
		\item $r :\mathcal{X} \rightarrow \R^{n\times \tilde{r} }$ satisfies the following Lipschitz assumption: there exists a constant $ L_r\in\R^{+}_0$ such that $ \| r(\x_t)-r(\hat{\x}_t) \| \leq L_r \| \x_t-\hat{\x}_t \|_{[-\tau,0]}$ for all $ \x_t, \hat{\x}_t\in \mathcal{X} $.
	\end{compactitem}
\end{definition}
An $ \R^n $-valued continuous-time process $\xi$ is said to be a \textit{solution process} for $\Sigma_R$ if there exists $ \upsilon\in\mathcal{U} $ satisfying
\begin{align}
	\diff \xi(t) = f(\xi_t,\upsilon(t))\diff t+ g(\xi_t)\diff W_t+r(\xi_t)\diff P_t,
	\label{stocha_delay}
\end{align}
$ \mathbb{P} $-almost surely ($ \mathbb{P} $-a.s.), where $ f$, $g$, and $ r$ are the drift, diffusion, and reset terms, respectively, and $\xi_t:=\{\xi(t+\theta)| -\tau \leq \theta \leq 0\}$. We
emphasize that postulated assumptions on $ f$, $g$, and $ r$ ensure the existence and uniqueness of the solution process $\xi$ on $ t\geq -\tau $ \cite[Theorem 1.19]{Oks_jump}. Throughout the paper we use the notation $\xi_{\zeta,\upsilon}(t)$ to denote the value of a solution process starting from initial condition $ \zeta =\{\xi(\theta)|-\tau\leq\theta\leq 0\}\in \mathcal{C}^b_{\mathcal{F}_0}([-\tau, 0];\R^n)$ $ \mathbb{P} $-a.s. and under the input signal $ \upsilon $ at time $t$. We also use the notation $\xi_{t,\zeta,\upsilon}$ to denote the solution process starting from initial condition $ \zeta =\{\xi(\theta)|-\tau\leq\theta\leq 0\}\in \mathcal{C}^b_{\mathcal{F}_0}([-\tau, 0];\R^n)$ $ \mathbb{P} $-a.s. and under the input signal $ \upsilon $. Note that for any $ t \in \mathbb{R}^+_0 $, $\xi_{\zeta,\upsilon}(t)$ is a random variable taking values in $ \R^n $ and $\xi_{t,\zeta,\upsilon}$ is a random variable taking values in $ \mathcal{C}([-\tau, 0];\R^n) $. Here, we assume that the Poisson processes $P^{i}_s$, for any $i \in \{1, 2,\ldots, \tilde{r}\}$, have the rates of $  \lambda_{i}$. Now we will introduce delayed jump-diffusion system ($\Sigma_D$)(DJDS) as a special case of retarded jump-diffusion system which is given by 	
\begin{align}
	\diff \xi(t)\hspace{-.2em} =& F(\xi(t), \xi(t\hspace{-.1em}-\hspace{-.1em}\tau_1),\upsilon(t))\diff t \hspace{-.2em}+ \hspace{-.2em}G(\xi(t),\xi(t\hspace{-.1em}-\hspace{-.1em}\tau_2))\diff W_t+R(\xi(t),\xi(t\hspace{-.1em}-\hspace{-.1em}\tau_3))\diff P_t,
	\label{stocha_delay1}
\end{align}
where $ F: \R^n \times \R^n \times \mathsf{U} \rightarrow \R^n$, $G: \R^n \times \R^n \rightarrow \R^{n\times \overline{r}}$, and $ R :\R^n \times \R^n \rightarrow \R^{n\times \tilde{r} }$ are the drift, diffusion, and reset terms, respectively. The constants $ \tau_1 $, $ \tau_2 $, and $\tau_3$ are the state delay in the drift, diffusion, and reset terms, respectively.

\subsection{Incremental Stability for RJDS and DJDS}
Here, we introduce a notion of incremental stability for RJDS (resp. DJDS).
\begin{definition} \label{definition3} 
	An RJDS $\Sigma_R$ (resp. DJDS $\Sigma_D$) is incrementally input-to-state stable in the $ q^{th} $ moment, where $ q\geq1 $, denoted by ($ \delta$-ISS-M$_q $), if there exist a $ \mathcal{KL} $ function $ \beta $ and a $ \mathcal{K}_\infty $ function $ \gamma $ such that for any $ t \in \mathbb{R}^+_0 $, any two initial conditions $ \zeta, \hat{\zeta} \in \mathcal{C}^b_{\mathcal{F}_0}([-\tau, 0];\R^n)$, and any $ \upsilon, \hat{\upsilon}\in \mathcal{U} $ the following condition is satisfied:
	\begin{align}
		\mathbb{E}[\|\xi_{\zeta,\upsilon}(t)\hspace{-.2em}-\hspace{-.2em}\xi_{\hat{\zeta},\hat{\upsilon}}(t)\|^q]\leq\hspace{-.2em}\beta(\mathbb{E}[\|\zeta-\hat{\zeta}\|_{[-\tau, 0]}^q],t)\hspace{-.2em}+\gamma(\|\upsilon-\hat{\upsilon}\|_\infty).
		\label{nnn}
	\end{align} 
\end{definition}
One can readily verify that in the absence of delay, Definition \ref{definition3} reduces to that of $ \delta$-ISS-M$_q $ for stochastic control systems in \cite[Definition 3.1]{zamani2014symbolic}.  

For later use, we provide the infinitesimal generators (denoted by operator $\mathcal{L}$) for an RJDS $ \Sigma_R $ and a DJDS $ \Sigma_D $ using It\^{o}'s differentiation \cite[equation (23)]{julius1}. Let function $ V: \R^n \times \R^n\rightarrow \R_0^+ $ be twice differentiable on $ \R^n \times \R^n \setminus \bigtriangleup $. 
The infinitesimal generator of $ V $ associated with an RJDS $\Sigma_R$ in (\ref{stocha_delay}) is an operator, denoted by $ \mathcal{L}V $, from $ \mathcal{C}([-\tau,0];\R^n)\times \mathcal{C}([-\tau,0];\R^n) $ to $ \R $, and $ \forall t\in\R^+_0, \forall \x_t,\h\x_t\in\mathcal{C}([-\tau,0];\R^n) $ and $\forall u,\hat{u}\in\mathsf{U}$ it is given by
\begin{align}
	\mathcal{L}V(\x_t,\hat{\x}_t,u,\hat u):=&\begin{bmatrix}\partial_x V & \partial_{\hat{x}} V\end{bmatrix}\begin{bmatrix}
		f(\x_t,u)\\ 
		f(\hat{\x}_t,\hat{u})
	\end{bmatrix}
	+\hspace{-.2em} \frac{1}{2}\mathsf{Tr}\bigg(\hspace{-.3em}\begin{bmatrix}
		g(\x_t)\\ 
		g(\hat{\x}_t)
	\end{bmatrix}\hspace{-.3em} \begin{bmatrix}
	g^T(\x_t) &\hspace{-0.8em} g^T(\hat{\x}_t)
\end{bmatrix}\hspace{-.3em}\begin{bmatrix}
\partial_{x,x}V & \hspace{-.5em}\partial_{x,\h x}V \\ 
\partial_{\h x,x}V & \hspace{-.5em}\partial_{\hat{x},\hat{x}}V
\end{bmatrix}\hspace{-.3em}\bigg)\nonumber\\
&+\sum_{i=1}^{\tilde{r}}\lambda_i\big(V(\x_t(0)+r(\x_t)e_i,\hat{\x_t}(0)+r(\hat{\x}_t)e_i)-V(\x_t(0),\hat{\x_t}(0))\big).
\label{infi_gen}
\end{align}
The infinitesimal generator of $ V $  associated with a DJDS $\Sigma_D$ in (\ref{stocha_delay1}) is an operator, denoted by $ \mathcal{L}V $, from $ \R^{8n}$ to $ \R $ and $ \forall x, \h x, y,\h y,z, \h z, p,\h p\in\R^n $, and $\forall u,\hat{u}\in\mathsf{U}$ it is given by
\begin{align}
	\mathcal{L}V(x,\hat{x}, y, \hat{y}, z, \hat{z},p, \hat{p}, u, \hat u):=&\begin{bmatrix}\partial_x V & \partial_{\hat{x}} V\end{bmatrix}\begin{bmatrix}
		F(x,y,u)\\ 
		F(\hat{x},\hat{y},\hat{u})
	\end{bmatrix}+\hspace{-.2em} \frac{1}{2}\mathsf{Tr}\bigg(\hspace{-.3em}\begin{bmatrix}
		G(x,z)\\ 
		G(\hat{x},\hat{z})
	\end{bmatrix}\hspace{-.3em} \begin{bmatrix}
	G^T(x,z) &\hspace{-0.8em} G^T(\hat{x},\hat{z})
\end{bmatrix}\hspace{-.3em}\begin{bmatrix}
\partial_{x,x}V & \hspace{-.5em}\partial_{x,\h x}V \\ 
\partial_{\h x,x}V & \hspace{-.5em}\partial_{\hat{x},\hat{x}}V
\end{bmatrix}\hspace{-.3em}\bigg)\nonumber\\
&+\sum_{i=1}^{\tilde{r}}\lambda_i\big(V(x+R(x,p)e_i,\hat{x}+R(\hat{x},\hat{p})e_i)-V(x,\hat{x})\big).
\label{infi_gen1}
\end{align}
The symbols $ \partial_{x} $ and $ \partial_{x,\hat{x}} $ in (\ref{infi_gen}) and (\ref{infi_gen1}) represent first and second-order partial derivatives with respect to $ x $ (1st argument) and $ \hat{x} $ (2nd argument), respectively. Note that we dropped the arguments of  $ \partial_{x}V $, $ \partial_{\hat{x}}V $, $ \partial_{x,x}V $, $ \partial_{\hat{x},x}V $, $ \partial_{x,\hat{x}}V $, and  $ \partial_{\hat{x},\hat{x}}V $  in (\ref{infi_gen}) and (\ref{infi_gen1}) for the sake of simplicity.\\
Now we describe $ \delta$-ISS-M$_q $ in terms of existence of so-called $ \delta$-ISS-M$_q $ Lyapunov functions for RJDS and DJDS using Razumikhin-type condition as defined next.

\begin{definition}\label{definition4}
	Consider an RJDS $\textstyle \Sigma_R$ and a continuous function $ V : \R^n \times \R^n\rightarrow\R^+_0 $ that is twice differentiable on $  {\R^n \times \R^n} \setminus \bigtriangleup $. The function $ V $ is called a $ \delta$-ISS-M$_q $ Lyapunov function for $\textstyle \Sigma_R $ for $ q\geq1 $, if there exist $ \mathcal{K}_\infty $ functions  $ \underline{\alpha}, \overline{\alpha}$, and $ \varphi $, such that:
	\begin{enumerate}
		\item[(i)] $ \underline{\alpha} $(resp. $ \overline{\alpha} $) is a convex (resp. concave) function;
		\item[(ii)] $  \forall x, \hat{x}\in\R^n$, $\underline{\alpha}(\|x-\hat{x}\|^q) \leq V(x, \hat{x})\leq\overline{\alpha}(\|x-\hat{x}\|^q) $; 
		\item[(iii)] $ \forall u,\h u\in\mathsf{U} $ and $  \forall t\geq0$, 
		\begin{align}\label{con1}
		 \EE[\mathcal{L}V(\phi,\h \phi,u, \hat u)]\leq-\EE[\omega(\phi(0),\h \phi(0))]+\varphi(\|u-\hat{u}\|),
		\end{align}
		  for all $ \phi,\hat{\phi}\in \mathbf{L}_{\mathcal F_t}^q([-\tau,0];\R^n) $ satisfying 
		 \begin{align}\label{con2}
		 \EE[V(\phi(\theta),\hat{\phi}(\theta))]\leq\EE[\tilde q(\phi(0),\hat{\phi}(0))], \ \forall \theta\in[-\tau,0];
		 \end{align}
	\end{enumerate}
	where $\omega: \R^{n}\times\R^{n}\rightarrow\R^{+}$ is a nonnegative function such that there exists a $\mathcal{K}_\infty$ function $ \tilde{\omega} $ satisfying $\omega(x,\h x)\geq\tilde{\omega}(\|x-\h x\|^q)$ and $\lim_{\|s\|\rightarrow\infty}\frac{\tilde{\omega}(\|s\|^q)}{\overline{\alpha}(\|s\|^{q})}>0$; $\tilde q:\R^{n}\times\R^{n}\rightarrow\R$ is a function such that $\tilde q(x,\hat{x})-V(x,\hat{x})\geq\overline{q}(\|x-\hat{x}\|)$, where $\overline{q}$ is a $\mathcal{K}_\infty$ function satisfying $\lim_{\|s\|\rightarrow\infty}\frac{\overline{q}(\|s\|)}{\overline{\alpha}(\|s\|^{q})}>0$.
\end{definition}

\begin{definition}\label{definition5}
	Consider a DJDS $\textstyle \Sigma_D$ and a continuous function $ V :\R^n\times\R^n\ra\R^+_0$ that is twice differentiable on $ \R^n\times\R^n\setminus \bigtriangleup $. Function $V$ is called a $ \delta$-ISS-M$_q $ Lyapunov function for $\textstyle \Sigma_D $ for $ q\geq1 $, if there exist constants $\kappa_0$, $\kappa_1$, $\kappa_2$, $\kappa_3$ such that $ \kappa_0\geq \sum_{i=1}^{3}\kappa_i\geq 0 $, a nonnegative function $\psi:\R^n\times\R^n\rightarrow\R^+$, $ \mathcal{K}_\infty $ functions $ \ul\alpha, \ol\alpha, \varphi $, and $ \mathcal{K} $ function $ \hat{\omega} $ such that: conditions (i) and (ii) in Definition \ref{definition4} hold and $\forall x, \hat{x},y,\h y,z,\h z,p, \h p\in\R^n$ and $\forall u, \hat{u}\in\mathsf{U}$,
	\begin{align}
		\mathcal{L}V(x, \hat{x},y, \h y, z,\h z,p,\h p,u,\h u)\leq&-\kappa_0V(x,\hat{x})-\psi(x,\hat{x})+\kappa_1V(y,\hat{y})+\kappa_2V(z,\hat{z})+\kappa_3V(p,\hat{p})+\varphi(\|u-\hat{u}\|),\nonumber
	\end{align}
and $ \psi(x,\hat{x})\geq\hat{\omega}(\|x-\hat{x}\|^q) $ and $ \lim_{\|s\|\rightarrow\infty}\frac{\hat{\omega}(\|s\|^q)}{\overline{\alpha}(\|s\|^q)}$ $> 0$.
\end{definition}
Now we provide the description of $ \delta$-ISS-M$_q $ for an RJDS $ \Sigma_R $ in terms of existence of $ \delta$-ISS-M$_q $ Lyapunov functions in the following theorem.
\begin{theorem}\label{theorem1}
	An RJDS $\textstyle \Sigma_R $ is $ \delta$-ISS-M$_q $ if it admits a $ \delta$-ISS-M$_q $ Lyapunov function as in Definition \ref{definition4}.
\end{theorem}

\begin{proof}
	The proof is inspired by the proof of Theorem 3.1 in \cite{huang2009input}. Denote $ \overline{\varphi}=\varphi(\|\upsilon-\hat{\upsilon}\|_\infty) $ and $ \overline{V}_0=\underline{\alpha}(\EE[\|\zeta-\hat{\zeta}\|_{[-\tau,0]}^q])$, $ \forall \upsilon,\h\upsilon\in\mathcal{U} $ and $\forall\zeta, \hat{\zeta} \in \mathcal{C}^b_{\mathcal{F}_0}([-\tau, 0];\R^n)$. By using Lemma 3.2 and 3.3 in \cite{huang2009input}, there exist a constant $a_q>0 $ and a $ \mathcal{K}_\infty $ function $ \mu_\omega $ such that $ \forall t\geq0 $, $ \EE[\omega(\xi_{\zeta,\upsilon}(t),\xi_{\hat{\zeta},\hat{\upsilon}}(t))]\geq 2\overline{\varphi}$ and $ \EE[ \tilde q(\xi_{\zeta,\upsilon}(t),\xi_{\hat{\zeta},\hat{\upsilon}}(t))]-\EE [V(\xi_{\zeta,\upsilon}(t),\xi_{\hat{\zeta},\hat{\upsilon}}(t))]\geq a_q$, whenever $ \EE [V(\xi_{\zeta,\upsilon}(t),\xi_{\hat{\zeta},\hat{\upsilon}}(t))]\geq\mu_\omega^{-1}(2\overline{\varphi}) $. Without loss of generality, assume $ \mu_\omega^{-1}(2\overline{\varphi})<\underline{\alpha}(\sup_{-\tau\leq\theta\leq0}\EE[\|\zeta(\theta)-\hat{\zeta}(\theta)\|^q)]\leq \overline{V}_0$. Let $ J $ be the minimal nonnegative integer such that $ M_0=\mu_\omega^{-1}(2\overline{\varphi})+Ja_q>\overline{V}_0 $. Let $ \hat{\tau}=\max\{\tau,M_0/\overline{\varphi}\} $ and $ t_j=j\hat{\tau} $ for $ j\in\{0,1,\ldots ,J\} $. In order to prove the theorem, we need to show
	\begin{align}
		\EE[ V(\xi_{\zeta,\upsilon}(t),\xi_{\hat{\zeta},\hat{\upsilon}}(t))]\leq\min\{\overline{V}_0,M_j\},   \forall t\geq t_j,
		\label{p0}
	\end{align}
	where $ M_j=\mu_\omega^{-1}(2\overline{\varphi})+(J-j)a_q$ and $ j\in\{0,1,\ldots ,J\} $.
	First we show that $\EE [V(\xi_{\zeta,\upsilon}(t),\xi_{\hat{\zeta},\hat{\upsilon}}(t))]\leq\overline{V}_0, \forall t\geq t_0$.\\ 
	Suppose that $ t_a:=\inf\{t>t_0\st\EE [V(\xi_{\zeta,\upsilon}(t),\xi_{\hat{\zeta},\hat{\upsilon}}(t))]>\overline{V}_0\}<\infty $. Since $ \EE [V(\xi_{\zeta,\upsilon}(t),\xi_{\hat{\zeta},\hat{\upsilon}}(t))] $ is continuous in time $ t\geq 0 $, there exist a pair of constants $ t_b $ and $ t_c $ such that $ t_0\leq t_b\leq t_a<t_c $ and 
	\begin{align}
		\begin{matrix}
			\EE [V(\xi_{\zeta,\upsilon}(t),\xi_{\hat{\zeta},\hat{\upsilon}}(t))]=\overline{V_0}, &t=t_b; \\ 
			\overline{V_0}< \EE[V(\xi_{\zeta,\upsilon}(t),\xi_{\hat{\zeta},\hat{\upsilon}}(t))] <\overline{V_0}+a_q, & t_b < t \leq t_c.
		\end{matrix}
		\label{p1}
	\end{align}
	However, by generalized It\^{o}'s formula \cite{skorokhod2009asymptotic} and condition (\ref{con1}) in Definition \ref{definition4}, we have
	\begin{align}
		&\EE [V(\xi_{\zeta,\upsilon}(t),\xi_{\hat{\zeta},\hat{\upsilon}}(t))]=\EE [V(\xi_{\zeta,\upsilon}(t_b),\xi_{\hat{\zeta},\hat{\upsilon}}(t_b))]+\int_{t_b}^{t}\EE[\mathcal{L}V(\xi_{s,\zeta,\upsilon},\xi_{s,\h \zeta,\h \upsilon})]\diff s\leq \overline{V_0}- \overline{\varphi}(t-t_b)  \leq \overline{V_0}\nonumber
	\end{align}
	fot all $ t\in(t_b,t_c] $, which contradicts (\ref{p1}). Thus the inequality $\EE [V(\xi_{\zeta,\upsilon}(t),\xi_{\hat{\zeta},\hat{\upsilon}}(t))]\leq\overline{V}_0$ must be true for all $ t\geq t_0$.
	Now we show that $\EE [V(\xi_{\zeta,\upsilon}(t),\xi_{\hat{\zeta},\hat{\upsilon}}(t))]\leq M_1, \forall t\geq t_1$. Let $ t_m:= \inf\{t\geq t_0\st \EE [V(\xi_{\zeta,\upsilon}(t),\xi_{\hat{\zeta},\hat{\upsilon}}(t))]\leq M_1\} <\infty$. If $ t_m >t_1 $, then $ \forall t\in [t_0, t_1] $, we have 
	\begin{align}
		\EE[ q(\xi_{\zeta,\upsilon}(t),\xi_{\hat{\zeta},\hat{\upsilon}}(t))]\hspace{-.2em}&\geq\hspace{-.2em}\EE [V(\xi_{\zeta,\upsilon}(t),\xi_{\hat{\zeta},\hat{\upsilon}}(t))]\hspace{-.2em}+\hspace{-.2em}a_q\hspace{-.2em}>\hspace{-.2em}M_1\hspace{-.2em}+\hspace{-.2em}a_q\hspace{-.2em}>\hspace{-.2em}\overline{V}_0\geq\hspace{-.2em} \EE[ V(\xi_{\zeta,\upsilon}(t+\theta),\xi_{\h\zeta,\h\upsilon}(t+\theta))],\forall \theta\in[-\tau,0].\label{aaa}
	\end{align}  
	Using condition (\ref{con1}) in Definition \ref{definition4}, inequality (\ref{aaa}) implies 
	\[ \EE [\mathcal{L}V(\xi_{t,\zeta,\upsilon}, \xi_{t,\hat{\zeta},\hat{\upsilon}})]\leq -\overline{\varphi} , \ \forall t\in[t_0,t_1].\]
	Consequently, by generalized It\^{o}'s formula, we have $ \EE [V(\xi_{\zeta,\upsilon}(t)$,$\xi_{\hat{\zeta},\hat{\upsilon}}(t))]\leq \overline{V}_0- \overline{\varphi} \hat{\tau}<0$, which contradicts the property of $ \EE [V(\xi_{\zeta,\upsilon}(t)$, $\xi_{\hat{\zeta},\hat{\upsilon}}(t))]\geq 0, \ \forall t\geq 0 $. Hence, we must have $ t_m\leq t_1 $. Let
	\[\bar{t}_a:=\inf\{t>t_m\st\EE[ V(\xi_{\zeta,\upsilon}(t),\xi_{\hat{\zeta},\hat{\upsilon}}(t))]>M_1\} <\infty.\]
	Again as $ \EE[ V(\xi_{\zeta,\upsilon}(t),\xi_{\hat{\zeta},\hat{\upsilon}}(t))] $ is continuous in $ t\geq0 $, there exists constants $ \bar{t}_b $ and $ \bar{t}_c $ such that $ t_1\leq \bar{t}_b\leq \bar{t}_a< \bar{t}_c $ and 
	\begin{align}
		\begin{matrix}
			\EE[ V(\xi_{\zeta,\upsilon}(t),\xi_{\hat{\zeta},\hat{\upsilon}}(t))]=M_1, &t=\bar{t}_b; \\ 
			M_1< \EE [V(\xi_{\zeta,\upsilon}(t),\xi_{\hat{\zeta},\hat{\upsilon}}(t))] <M_1+a_q, & \bar{t}_b < t \leq \bar{t}_c.
		\end{matrix}\nonumber
	\end{align}
	By using similar reasoning as before, generalized It\^{o}'s formula \cite{skorokhod2009asymptotic} and condition (\ref{con1}) in Definition \ref{definition4}, the assumption results in contradiction, thus we have (\ref{p0}) for $ j=1 $. Now define $ t_j:= \inf\{t\geq t_{j-1}$  $ | $  $\EE [V(\xi_{\zeta,\upsilon}(t),\xi_{\hat{\zeta},\hat{\upsilon}}(t))]\leq M_j\}<\infty $ for $ j=2,3,\ldots ,J $. By similar type of reasoning, we get $ \EE [V(\xi_{\zeta,\upsilon}(t),\xi_{\hat{\zeta},\hat{\upsilon}}(t))]\leq M_j$, $\forall t\geq t_j $. Particularly, $ \EE [V(\xi_{\zeta,\upsilon}(t),\xi_{\hat{\zeta},\hat{\upsilon}}(t))]\leq M_J= \mu_\omega^{-1}(2\overline{\varphi}), \forall t\geq t_J $. By following Jensen's inequality, one obtains
	\begin{align}
		\EE[\|\xi_{\zeta,\upsilon}(t)-\xi_{\hat{\zeta},\hat{\upsilon}}(t)\|^q] \leq\gamma(\|\upsilon-\hat{\upsilon}\|_\infty), \ \forall t\geq t_J,
		\label{p3} 
	\end{align}
	where $ \gamma(s)=\underline{\alpha}^{-1}(\mu_\omega^{-1}(2\varphi(s))) $ for all $ s \in \R_0^+ $. 
	Now choose a $ \mathcal{KL} $ function $ \overline{\beta} $ such that $ \overline{\beta}(\overline{V}_0,t)\geq 2 \overline{V}_0-\frac{\overline{V}_0}{t_J}t $, $ \forall t\in[0,t_J] $. So we have $ \EE 
	[V(\xi_{\zeta,\upsilon}(t),\xi_{\hat{\zeta},\hat{\upsilon}}(t))] \leq \overline{\beta}(\overline{V}_0,t),\forall t\in [0,t_J] $ which implies
	\begin{align}
		\EE[\|\xi_{\zeta,\upsilon}(t)-\xi_{\hat{\zeta},\hat{\upsilon}}(t)\|^q] \leq\beta(\mathbb{E}[\|\zeta-\hat{\zeta}\|_{[-\tau, 0]}^q],t), \ \forall t\in [0,t_J],
		\label{p4} 
	\end{align}
	where $ \beta(s,t)=\underline{\alpha}^{-1}(\overline{\beta}(\underline{\alpha}(s),t)) $ for any $ s,t\in\R^+_0 $. 
	From (\ref{p3}) and (\ref{p4}), one can readily verify inequality (\ref{nnn}) which implies that $\textstyle \Sigma_R $ is $ \delta$-ISS-M$_q $.
\end{proof}
The next corollary proposes similar results as in the previous theorem but for DJDS. 
\begin{corollary}\label{col1}
	A DJDS $\textstyle \Sigma_D $ is $ \delta$-ISS-M$_q $ if it admits a $ \delta$-ISS-M$_q $ Lyapunov function as in Definition \ref{definition5}.\\
\end{corollary}
\begin{proof}
	Let $ \omega(x,\hat{x})=\frac{1}{1+\kappa_0}\psi(x,\hat{x}) $ for all $ x,\h x\in\R^n $. Now by considering Definition \ref{definition5}, we have
	\begin{align}
		\EE[\mathcal{L}V(\phi,\h \phi,u,\h u)]\hspace{-.2em}=&\EE[\mathcal{L}V(\phi(0),\h\phi(0),\phi(-\tau_1),\h\phi(-\tau_1),\phi(-\tau_2),\h\phi(-\tau_2),\phi(-\tau_3),\h\phi(-\tau_3))]\nonumber\\
		\leq&-\kappa_0\EE [V(\phi(0),\h\phi(0))]-\EE[\psi(\phi(0),\h \phi(0))]
		+\kappa_1\EE[V(\phi(-\tau_1),\h \phi(-\tau_1))]+\kappa_2\EE[V(\phi(-\tau_2),\h \phi(-\tau_2))]\nonumber\\&+\kappa_3\EE[V(\phi(-\tau_3),\h \phi(-\tau_3))]+\varphi(\|u-\hat{u}\|)\nonumber\\
		\leq&-\kappa_0\big(\EE[ V(\phi(0),\h \phi(0))]+ \EE[\omega(\phi(0),\h \phi(0))]\big)+\kappa_1\EE[V(\phi(-\tau_1),\h \phi(-\tau_1))]+\kappa_2\EE[V(\phi(-\tau_2),\h \phi(-\tau_2))]\nonumber\\&+\kappa_3\EE[V(\phi(-\tau_3),\h \phi(-\tau_3))]-\EE[\omega(\phi(0),\h \phi(0))]+\varphi(\|u-\hat{u}\|)\nonumber\\
		\leq&-(\kappa_0-\sum_{i=1}^{3}\kappa_i)\Big(\EE [V(\phi(0),\h \phi(0))]+ \EE[\omega(\phi(0),\h \phi(0))]\Big)-\EE[\omega(\phi(0),\h \phi(0))]+\varphi(\|u-\hat{u}\|)\nonumber\\
		\leq& -\EE[\omega(\phi(0),\h \phi(0))]+\varphi(\|u-\hat{u}\|)\nonumber
	\end{align}	
	  for all $ t\geq0 $ and $ \phi,\h \phi\in\mathbf{L}_{\mathcal F_t}^q([-\tau,0];\R^n) $ satisfying condition (\ref{con2}) in Definition \ref{definition4} with function $ \tilde{q}(\phi(0),\hat{\phi}(0)):=V(\phi(0),\hat{\phi}(0))+\omega(\phi(0),\hat{\phi}(0)) $. Moreover, functions $ \tilde{\omega}(s)=\overline{q}(s)=\frac{\hat{\omega}(s)}{1+\kappa_0} $, $\forall s\in\R_0^+$, satisfy properties required in condition (iii) in Definition \ref{definition4}. Therefore, $ V $ satisfies all the conditions in Definition \ref{definition4}. Thus by following Theorem \ref{theorem1}, we obtain that $\textstyle \Sigma_D $ is $ \delta$-ISS-M$_q $.
\end{proof}
In the following lemma, we provide a similar result as in Corollary \ref{col1} but tailored to linear delayed jump-diffusion systems in which sufficient conditions boil down to a matrix inequality.  
\begin{lemma}\label{lem3}
	Consider a DJDS $ \Sigma_D $ as given in (\ref{stocha_delay1}), where for all $ x, y, z,p \in\R^n $ and $ u\in\mathsf{U}$, $ F(x,y,u):=A_1x+A_2y+Bu $, for some $ A_1,A_2\in\R^{n\times n} $ and $  B\in\R^{n\times m} $, $ G(x,z):=[G_1x $ $ G_2x $ $ \cdots $ $ G_{\ol r}x] +[\ol G_1z $ $ \ol G_2z $ $ \cdots $ $ \ol G_{\ol r}z]$ and $R(x,p):=[R_1x $ $ R_2x $ $ \cdots $ $ R_{\tilde{r}}x] +[\ol R_1p $ $ \ol R_2p $ $ \cdots $ $ \ol R_{\tilde{r}}p]$, for some $ G_i,\ol G_i, R_i,\ol R_i\in\R^{n\times n} $. Then, system $ \Sigma_D $ is $ \delta$-ISS-M$_2 $ if there exist constants $  c_1, c_2, c_3, c_4, c_5\in\R^+ $ satisfying $ c_1>\sum_{i=2}^{4}c_i $ and
	\begin{align}\label{one}
	\begin{bmatrix}
	\Delta & \hspace{-.3em}PA_2 & \hspace{-.2em}\sum\limits_{i=1}^{\bar{r}}G_i^TP\overline{G}_i & \sum\limits_{i=1}^{\tilde{r}}\lambda_i(P\overline{R}_i\hspace{-.2em}+\hspace{-.2em}R_i^TP\overline{R}_i) & \hspace{-.3em}PB\\ 
	A_2^TP & 0 & 0 & 0 & 0\\ 
	\sum\limits_{i=1}^{\bar{r}}\overline{G}_i^TPG_i& 0 & \sum\limits_{i=1}^{\bar{r}}\overline{G}_i^TP\overline{G}_i & 0 & 0\\ 
	\sum\limits_{i=1}^{\tilde{r}}\lambda_i(\overline{R}_i^TP\hspace{-.2em}+\hspace{-.2em}\overline{R}_i^TPR_i) & 0 & 0 & \sum\limits_{i=1}^{\tilde{r}}\lambda_i\overline{R}_i^TP\overline{R}_i  & 0\\ 
	B^TP & 0 & 0 & 0 & 0
	\end{bmatrix}\preceq\begin{bmatrix}
	-c_1P & 0 & 0 &0  & 0\\ 
	0 & \hspace{-.3em}c_2P & 0 & 0 & 0\\ 
	0 & 0 &\hspace{-.3em} c_3P & 0 &0\\ 
	0 & 0 & 0 &\hspace{-.3em}c_4P  & 0\\ 
	0 & 0 & 0 & 0 &\hspace{-.3em}c_5I_m 
	\end{bmatrix}
	\end{align} 
	where $ P $ is a symmetric positive definite matrix and $\Delta=  PA_1+A_1^TP+\sum\limits_{i=1}^{\bar{r}}G_i^TPG_i+\sum\limits_{i=1}^{\tilde{r}}\lambda_i(PR_i+R_i^TP+R_i^TPR_i)$.
\end{lemma}
\begin{proof}
	Consider a function $ V :\R^n\times\R^n\ra\R^+_0$ given by
	\begin{align}\label{lyap_1}
		V(x,\h x):=\frac{1}{2}(x-\h x)^TP(x-\h x),\ \ \ \forall x,\h x\in\R^n,
	\end{align}
	where $ P $ is a symmetric positive definite matrix.
	One can readily verify that the function V in (\ref{lyap_1}) satisfies properties (i) and (ii) in Definition \ref{definition4} with functions $ \ul \alpha(s):=\frac{1}{2}\lambda_{\min}(P)s $ and $ \ol \alpha(s):=\frac{1}{2}\lambda_{\max}(P)s $ for all $ s\in\R^+_0 $ and $q=2$. By considering the infinitesimal generator in (\ref{infi_gen1}) associated with the considered linear delayed jump-diffusion system, Lipschitz assumptions, Young's inequality, consistency of norms, and (\ref{one}) one can obtain the following chains of inequalities
	\begin{align}
		\mathcal{L}&V(x,\h{x}, y,\h{y},z,\h{z},p,\h{p},u,\h u)=(x-\h x)^TP(A_1(x-\h x)+A_2(y-\h y)+B(u-\h u))\nonumber\\&+\frac{1}{2}\sum_{i=1}^{\ol r}\Big(G_i(x-\h x)+\ol G_i(z-\h z)\Big)^TP\Big(G_i(x-\h x)+\ol G_i(z-\h z)\Big)+ 
		\frac{1}{2}\sum_{i=1}^{\tilde{r}}\lambda_i\Big[\Big((x-\h x)+R_i(x-\h x)+\ol R_i(p-\h p)\Big)^T\nonumber\\&P\Big((x-\h x)+R_i(x-\h x)+\ol R_i(p-\h p)\Big)-(x-\h x)^TP(x-\h x)\Big]\nonumber\\
		\leq&(x-\h x)^TP(A_1(x-\h x)+A_2(y-\h y)+B(u-\h u))+\frac{1}{2}\sum_{i=1}^{\ol r}\Big[(x-\h x)^TG_i^TPG_i(x-\h x)+(x-\h x)^TG_i^TP\ol G_i(z-\h z)\nonumber\\&+(z-\h z)^T\ol G_i^TPG_i(x-\h x)+(z-\h z)^T\ol G_i^TP\ol G_i(z-\h z)\Big]+ 
		\frac{1}{2}\sum_{i=1}^{\tilde{r}}\lambda_i\Big[(x-\h x)^TP(R_i(x-\h x)+\ol R_i(p-\h p))\nonumber\\&+(R_i(x-\h x)+\ol R_i(p-\h p))^TP(x-\h x)+(R_i(x-\h x)+\ol R_i(p-\h p))^TP(R_i(x-\h x)+\ol R_i(p-\h p))\Big]\nonumber\\
		\leq&(x-\h x)^TP(A_1(x-\h x)+A_2(y-\h y)+B(u-\h u))+\frac{1}{2}\sum_{i=1}^{\ol r}\Big[(x-\h x)^TG_i^TPG_i(x-\h x)+(x-\h x)^TG_i^TP\ol G_i(z-\h z)\nonumber\\
		&+(z-\h z)^T\ol G_i^TPG_i(x-\h x)+(z-\h z)^T\ol G_i^TP\ol G_i(z-\h z)\Big]+\frac{1}{2}\sum_{i=1}^{\tilde{r}}\lambda_i\Big[(x-\h x)^TP(R_i(x-\h x)+\ol R_i(p-\h p))\nonumber\\
		&+((x-\h x)^TR_i^T+(p-\h p)^T\ol R_i^T)P(x-\h x)+(x-\h x)^TR_i^TPR_i(x-\h x)+(x-\h x)^T R_i^TP\ol R_i(p-\h p)\nonumber\\
		&+(p-\h p)^T\ol R_i^TPR_i(x-\h x)+(p-\h p)^T\ol R_i^TP\ol R_i(p-\h p))\Big]\nonumber\\
		\leq&\frac{1}{2}\begin{bmatrix}
		x-\h x\\
		y-\h y\\
		z-\h z\\
		p-\h p\\
		u-\h u
		\end{bmatrix}^T
		\begin{bmatrix}
			\Delta & \hspace{-.3em}PA_2 & \hspace{-.2em}\sum\limits_{i=1}^{\bar{r}}G_i^TP\overline{G}_i & \sum\limits_{i=1}^{\tilde{r}}\lambda_i(P\overline{R}_i\hspace{-.2em}+\hspace{-.2em}R_i^TP\overline{R}_i) & \hspace{-.3em}PB\\ 
			A_2^TP & 0 & 0 & 0 & 0\\ 
			\sum\limits_{i=1}^{\bar{r}}\overline{G}_i^TPG_i& 0 & \sum\limits_{i=1}^{\bar{r}}\overline{G}_i^TP\overline{G}_i & 0 & 0\\ 
			\sum\limits_{i=1}^{\tilde{r}}\lambda_i(\overline{R}_i^TP\hspace{-.2em}+\hspace{-.2em}\overline{R}_i^TPR_i) & 0 & 0 & \sum\limits_{i=1}^{\tilde{r}}\lambda_i\overline{R}_i^TP\overline{R}_i  & 0\\ 
			B^TP & 0 & 0 & 0 & 0
		\end{bmatrix}
		\begin{bmatrix}
		x-\h x\\
		y-\h y\\
		z-\h z\\
		p-\h p\\
		u-\h u
		\end{bmatrix}\nonumber\\
		\leq&\frac{1}{2}\big(-c_1(x-\h x)^TP(x-\h x)+c_2(y-\h y)^TP(y-\h y)+c_3(z-\h z)^TP(z-\h z)+c_4(p-\h p)^TP(p-\h p)+c_5\|u-\h u\|^2\big)\nonumber\\
		\leq&-c_1V(x,\h x)+c_2V(y,\h y)+c_3V(z,\h z)+c_4V(p,\h p)+\frac{c_5}{2}\|u-\h u\|^2.\nonumber
	\end{align}
Thus by following the proof of Corollary \ref{col1} with $ \kappa_0=\sum_{i=2}^4c_i $, $ \psi(x,\h x)=(c_1-\kappa_0)V(x,\h x) $, $ \omega(x,\h x)= \frac{1}{1+\kappa_0}\psi(x,\h x)=\kappa V(x,\h x) $, where $ \kappa=\frac{c_1-\kappa_0}{1+\kappa_0} $, $\forall t\geq0 $, $\forall \zeta, \hat{\zeta} \in \mathcal{C}^b_{\mathcal{F}_0}([-\tau, 0];\R^n)$, and $\forall \phi,\h \phi\in\mathbf{L}_{\mathcal F_t}^k([-\tau,0];\R^n) $ satisfying (\ref{con2}), one obtains
	\begin{align}\label{xyz}
	\EE[\mathcal{L}V(\phi,\h \phi,u,\h u)]\leq-\kappa\EE[V(\xi_{\zeta,\upsilon}(t),\xi_{\h\zeta,\h\upsilon}(t))]+\frac{c_5}{2}\|u-\h u\|^2.
	\end{align}
	By using generalized Ito's formula, (\ref{xyz}), and condition (ii) in Definition \ref{definition4}, we have
	\begin{align*}
	\EE[V(\xi_{\zeta,\upsilon}(t),\xi_{\h\zeta,\h\upsilon}(t))]&=\EE[V(\xi_{\zeta,\upsilon}(0),\xi_{\h\zeta,\h\upsilon}(0))+\int_{0}^{t}\mathcal{L}V(\phi,\h\phi,u,\h u) \diff s]=\EE[V(\xi_{\zeta,\upsilon}(0),\xi_{\h\zeta,\h\upsilon}(0))]+\int_{0}^{t}\EE[\mathcal{L}V(\phi,\h\phi,u,\h u)] \diff s
	\nonumber\\&\leq\frac{\lambda_{\max}(P)}{2}\EE[\|\xi_{\zeta,\upsilon}(0)-\xi_{\h\zeta,\h\upsilon}(0)\|^2]+\int_{0}^{t}\EE[\mathcal{L}V(\phi,\h\phi)] \diff s\nonumber
	\nonumber\\&\leq\frac{\lambda_{\max}(P)}{2}\EE[\|\zeta-\h\zeta\|_{[-\tau,0]}^2]+\int_{0}^{t}\Big(-\kappa\EE[V(\xi_{\zeta,\upsilon}(s),\xi_{\h\zeta,\h\upsilon}(s))]+\frac{c_5}{2}\|\upsilon(s)-\h \upsilon(s)\|^2\Big) \diff s\nonumber\\
	&\leq\frac{\lambda_{\max}(P)}{2}\EE[\|\zeta-\h\zeta\|_{[-\tau,0]}^2]+-\kappa\int_{0}^{t}\EE[V(\xi_{\zeta,\upsilon}(s),\xi_{\h\zeta,\h\upsilon}(s))] \diff s+\frac{c_5}{2}\|\upsilon-\h \upsilon\|_\infty^2 t,
	\end{align*}
	which, by virtue of Gronwall's inequality, leads to
	\begin{align*}
	\EE[V(\xi_{\zeta,\upsilon}(t),\xi_{\h\zeta,\h\upsilon}(t))]&\leq \frac{\lambda_{\max}(P)}{2}\EE[\|\zeta-\h\zeta\|_{[-\tau,0]}^2]\e^{-\kappa t}+\frac{c_5t\e^{-\kappa t}}{2}\|\upsilon-\h \upsilon\|_\infty^2\\
	&\leq \frac{\lambda_{\max}(P)}{2}\EE[\|\zeta-\h\zeta\|_{[-\tau,0]}^2]\e^{-\kappa t}+\frac{c_5}{2\e\kappa}\|\upsilon-\h \upsilon\|_\infty^2.
	\end{align*}
	Now by using condition (ii) in Definition \ref{definition4}, one obtains
   \begin{align}
    \frac{\lambda_{\min}(P)}{2}\EE[\|\xi_{\zeta,\upsilon}(t)-\xi_{\h\zeta,\h\upsilon}(t)\|^2]&\leq\EE[V(\xi_{\zeta,\upsilon}(t),\xi_{\h\zeta,\h\upsilon}(t))]\leq \frac{\lambda_{\max}(P)}{2}\EE[\|\zeta-\h\zeta\|_{[-\tau,0]}^2]\e^{-\kappa t}+\frac{c_5}{2\e\kappa}\|\upsilon-\h \upsilon\|_\infty^2,\nonumber
    \end{align}
    and, hence,
    \begin{align}
    \EE[\|\xi_{\zeta,\upsilon}(t)-\xi_{\h\zeta,\h\upsilon}(t)\|^2]&\leq \frac{\lambda_{\max}(P)}{\lambda_{\min}(P)}\EE[\|\zeta-\h\zeta\|_{[-\tau,0]}^2]\e^{-\kappa t}+\frac{c_5}{\lambda_{\min}(P)\e\kappa}\|\upsilon-\h \upsilon\|_\infty^2.\nonumber
    \end{align}
	 Therefore, by introducing functions $ \beta $ and $ \gamma $ as
	 \begin{align}
	 \beta(s,t)=\frac{\lambda_{\max}(P)}{\lambda_{\min}(P)}\e^{-\kappa t}s, \gamma(s)=\frac{c_5}{\lambda_{\min}(P)\e\kappa}s^2,
	 \end{align}
	 for any $ s,t\in\R^+_0 $, inequality (\ref{nnn}) is satisfied. 
\end{proof}
\begin{remark}
For fixed values of $c_i$, $i=\{1,\ldots,5\}$, the inequality \eqref{one} boils down to a linear matrix inequality (LMI) which can be solved efficiently using semidefinite programming. One may also solve a bilinear matrix inequality (BMI) (locally) using a $V-K$ iteration \cite{411398}. That is, for fixed values $c_i$, $i=\{1,\ldots,5\}$, we find matrix $P$ satisfying the LMI, and then for a fixed $P$ we find constants $c_i$, $i=\{1,\ldots,5\}$, to maximize the value of $ c_1-\sum_{i=2}^{4}c_i $, and we iterate until there is no improvement in the value of $ c_1-\sum_{i=2}^{4}c_i $. 
\end{remark}

In order to provide results on the construction of symbolic models and given a RJDS $\Sigma_R$, we introduce the corresponding non-probabilistic retarded systems (denoted by $\overline{\Sigma}_R$) obtained by removing diffusion and reset terms (that is, $ g $ and $ r $ in (\ref{stocha_delay})). From now onwards, we use notation $\overline{\xi}_{\zeta,\upsilon}(t)$ to denote the value of a trajectory of $ \ol\Sigma_R $ in $ \R^n $ and $\overline{\xi}_{t,\zeta,\upsilon}$ to denote the solution of $ \ol\Sigma_R $ in $ \mathcal{C}([-\tau, 0];\R^n) $ at time $ t \in \mathbb{R}^+_0 $ started from the non-probabilistic initial condition $ \zeta \in \mathcal{C}^b_{\mathcal{F}_0}([-\tau, 0];\R^n)$, where $ \mathcal{F}_0 $ is the trivial sigma-algebra, and under input signal $ \upsilon $.
Now, we provide a technical lemma which is used later to show a relation between non-probabilistic retarded systems $\overline{\Sigma}_R$ and their symbolic models.
\begin{lemma}
	Consider an incrementally input-to-state stable non-probabilistic retarded systems $\overline{\Sigma}_R$ corresponding to a $ \delta$-ISS-M$_q $ RJDS $\Sigma_R$ for $q\geq1$, that is for any $ t \in \mathbb{R}^+_0 $, any $ \zeta, \hat{\zeta} \in \mathcal{C}([-\tau, 0];\R^n)$, and any $ \upsilon, \hat{\upsilon}\in \mathcal{U} $, it satisfies
	\begin{align}
	\|\ol\xi_{\zeta,\upsilon}(t)\hspace{-.2em}-\hspace{-.2em}\ol\xi_{\hat{\zeta},\hat{\upsilon}}(t)\|^q\leq\hspace{-.2em}\beta(\|\zeta-\hat{\zeta}\|^q_{[-\tau, 0]},t)\hspace{-.2em}+\gamma(\|\upsilon-\hat{\upsilon}\|_\infty),
	\label{nnn1}
	\end{align}
	where $\beta$ and $\gamma$ are the functions appearing in \eqref{nnn}.
	Then there exists a $ \mathcal{KL} $ function $ \tilde{\beta} $ such that the following inequality holds:
		\begin{align}
			\|\ol\xi_{t,\zeta,\upsilon}\hspace{-.2em}-\hspace{-.2em}\ol\xi_{t,\hat{\zeta},\hat{\upsilon}}\|^q_{[-\tau,0]}\hspace{-.2em}\leq\hspace{-.2em}\tilde{\beta}(\|\zeta-\hat{\zeta}\|^q_{[-\tau, 0]},t)\hspace{-.2em}+\hspace{-.2em}\gamma(\|\upsilon-\hat{\upsilon}\|_\infty),
			\label{mmm1}
		\end{align}
		where $ \tilde{\beta}(s,t)=\mathsf{e}^{-(t-\tau)}s+\beta(s,\max\{0,t-\tau\}) $.
\end{lemma}
\begin{proof}
	The proof is inspired by the proof of Theorem 3 in \cite{pola2010symbolic}. From inequality \eqref{nnn1}, we obtain the following inequalities: 
		\begin{align}
			\|\xi_{t,\zeta,\upsilon}-\xi_{t,\hat{\zeta},\hat{\upsilon}}\|^q_{[-\tau,0]}&\leq\beta(\|\zeta-\hat{\zeta}\|^q_{[-\tau, 0]},t-\tau)+\gamma(\|\upsilon-\hat{\upsilon}\|_\infty), \ \forall t\geq\tau,
			\label{l1}
		\end{align}
		and
		\begin{align}
			\|\xi_{t,\zeta,\upsilon}-\xi_{t,\hat{\zeta},\hat{\upsilon}}\|^q_{[-\tau,0]}&\leq\|\zeta-\hat{\zeta}\|^q_{[-\tau, 0]}\hspace{-.2em}+\hspace{-.2em}\beta(\|\zeta-\hat{\zeta}\|^q_{[-\tau, 0]},0)+\gamma(\|\upsilon-\hat{\upsilon}\|_\infty), \ \forall t\in[0,\tau).
			\label{l2}
		\end{align}
		Moreover, we also have 
		\begin{align}
			\mathsf{e}^{-(t-\tau)}\|\zeta-\hat{\zeta}\|^q_{[-\tau, 0]}\geq\|\zeta-\hat{\zeta}\|^q_{[-\tau, 0]}, \ \ \  \ \forall t\in[0,\tau).
			\label{l3}
		\end{align}	
		Inequalities (\ref{l1}) and (\ref{l2}) along with (\ref{l3}) yield
		\begin{align}
			\|\xi_{t,\zeta,\upsilon}-\xi_{t,\hat{\zeta},\hat{\upsilon}}\|^q_{[-\tau,0]}\leq\mathsf{e}^{-(t-\tau)}\|\zeta-\hat{\zeta}\|^q_{[-\tau, 0]}+\beta(\|\zeta-\hat{\zeta}\|^q_{[-\tau, 0]},\max\{0,t-\tau\})\hspace{-.2em}+\hspace{-.2em}\gamma(\|\upsilon-\hat{\upsilon}\|_\infty), \ \forall t\geq 0.
			\nonumber
		\end{align}
		One can rewrite the last inequality as 
		\begin{align}
			\|\xi_{t,\zeta,\upsilon}-\xi_{t,\hat{\zeta},\hat{\upsilon}}\|^q_{[-\tau,0]}\leq\tilde{\beta}(\|\zeta-\hat{\zeta}\|^q_{[-\tau, 0]},t)\hspace{-.2em}+\hspace{-.2em}\gamma(\|\upsilon-\hat{\upsilon}\|_\infty),
			\nonumber
		\end{align}
		for all $ t\geq0 $, where $ \tilde{\beta}(s,t):=\mathsf{e}^{-(t-\tau)}s+\beta(s,\max\{0,t-\tau\}) $ is a $ \mathcal{KL} $ function.
\end{proof}

\section{Systems and Approximate Equivalence Relations}  
We recall the notion of \textit{system} introduced in \cite{tabuada2009verification} which later serves as a unified modeling framework for both retarded jump-diffusion systems and their finite abstractions.
\begin{definition}\label{definition6}
	A systems is a tuple $ S=(X,X_0,U, \longrightarrow, Y,H) $ where $ X $ is a set of states (possibly infinite), $ X_0\subseteq X $ is a set of initial states, $ U $ is a set of inputs (possibly infinite), $ \longrightarrow \subseteq X\times U\times X $ is a transition relation, $ Y $ is a set of outputs, and $ H:X\rightarrow Y $ is an output map. 
\end{definition}
We denote $ x \overset{u}{\longrightarrow} x' $ as an alternative representation for a transition $ (x,u,x')\in\longrightarrow  $, where state $ x' $ is called a $ u $-successor (or simply successor) of state $ x $, for some input $ u\in U $.       
Moreover, a system $ S $ is said to be
\begin{compactitem}
	\item \textit{metric}, if the output set $ Y $ is equipped with a metric $ \mathbf{d}: Y\times Y \rightarrow\R^+_0 $.
	\item \textit{finite} (or \textit{symbolic}), if $ X $ and $ U $ are finite.
	\item \textit{deterministic}, if there exists at most a $ u $-successor of $ x $, for any $ x\in X $ and $ u\in U $.
	\item \textit{nonblocking}, if for any $ x\in X $, there exists some $ u $-successor of $ x $, for some $ u\in U $.    
\end{compactitem}
For a system $ S $, the finite state-run generated from initial state $ x_0\in X_0 $ is a finite sequence of transitions:
\begin{align}
	x_0 \overset{u_0}{\longrightarrow} x_1\overset{u_1}{\longrightarrow} \cdots \overset{u_{k-2}}{\longrightarrow} x_{k-1}\overset{u_{k-1}}{\longrightarrow} x_k,
\end{align}
such that $ x_i\overset{u_i}{\longrightarrow} x_{i+1} $, for $ i\in\{0,1,\ldots ,k-1\} $. The associated finite output-run is given by $ y_i=H(x_i) $, for $ i\in\{0,1,\ldots ,k-1\} $. These finite runs can be directly extended to infinite runs as well.\\
Now, we provide the notion of approximate (bi)simulation relation between two systems, introduced in \cite{Girard1}, which is later used for analyzing and synthesizing controllers for retarded jump-diffusion systems.
\begin{definition}\label{definition7}
	Let $ S_1=(X_1,X_{10},U_1,\underset{1}{\longrightarrow}, Y_1,H_1) $ and $ S_2=(X_2,X_{20},U_2,\underset{2}{\longrightarrow}, Y_2,H_2) $ be two metric systems having the same output sets $ Y_1=Y_2 $ and metric $ \dd $. For $ \varepsilon\in\R^+_0 $, a relation $ \mathcal{R}\subseteq X_1 \times X_2 $ is said to be an $ \varepsilon $-approximate bisimulation relation between $ S_1 $ and $ S_2 $ if it satisfies the following conditions:
	\begin{enumerate}
		\item[(i)] $ \forall (x_1,x_2)\in \mathcal{R} $, we have $ \dd(H_1(x_1),H_2(x_2))\leq\varepsilon $;
		\item[(ii)] $ \forall (x_1,x_2)\in \mathcal{R} $, $ x_1\overset{u_1}{\underset{1}{\longrightarrow}}x'_1 $ in $ S_1 $ implies $ x_2\overset{u_2}{\underset{2}{\longrightarrow}}x'_2 $ in $ S_2 $ satisfying $ (x'_1,x'_2)\in \mathcal{R} $;
		\item[(iii)] $ \forall (x_1,x_2)\in \mathcal{R} $, $ x_2\overset{u_2}{\underset{2}{\longrightarrow}}x'_2 $ in $ S_2 $ implies $ x_1\overset{u_1}{\underset{1}{\longrightarrow}}x'_1 $ in $ S_1 $ satisfying $ (x'_1,x'_2)\in \mathcal{R} $.
	\end{enumerate} 
	If we remove condition (iii), then $ \mathcal{R}\subseteq X_1 \times X_2 $ is said to be an $ \varepsilon $-approximate simulation relation from $ S_1 $ to $ S_2 $. 
\end{definition} 
The system $ S_1 $ is $ \varepsilon $-approximate bisimilar to $ S_2 $, denoted by $ S_1 \cong_\mathcal{S}^\varepsilon S_2 $, if there exists an $ \varepsilon $-approximate bisimulation relation $ \mathcal{R} $ between $ S_1 $ and $ S_2 $ such that: $ \forall x_{10}\in X_{10} $, $ \exists x_{20}\in X_{20} $ with $ (x_{10},x_{20})\in \mathcal{R} $ and $ \forall x_{20}\in X_{20} $, $ \exists x_{10}\in X_{10} $ with $ (x_{10},x_{20})\in \mathcal{R} $. 
In order to present the main results of the paper, we need to employ the notion of system as an abstract representation of a retarded jump-diffusion system. First, we define a metric system associated with the retarded jump-diffusion system $ \Sigma_R $, denoted by $ S(\Sigma_R) = (X,X_0,U,\longrightarrow, Y,H)$, where
\begin{compactitem}
	\item $ X $ is the set of all $ \mathcal{C}([-\tau, 0];\R^n) $-valued random variables defined on the probability space $(\Omega, \mathcal{F}, \mathds{P})$;
	\item $X_0$ is a subset of $\mathcal{C}^b_{\mathcal{F}_0}([-\tau, 0];\R^n)$;
	\item $ U=\mathcal{U} $;
	\item $\zeta\overset{\upsilon}{\longrightarrow} \zeta' $ if $ \zeta $ and $ \zeta' $ are measurable in $ \mathcal{F}_t $ and $ \mathcal{F}_{t+h} $, respectively, for some $ t\in\R^+_0 $ and $  h\in\R^+ $, and there exists a solution $\xi_{t}\in \mathbf{L}_{\mathcal{F}_t}^q([-\tau,0];\R^n) $ of $ \Sigma_R $ satisfying $ \xi_t=\zeta $ and $ \xi_{h,\zeta,\upsilon}=\zeta' $ $ \PP $-a.s.;
	\item $ Y=X$;
	\item $ H(\zeta)=\zeta$.
\end{compactitem}
From now on, we restrict our attention to the sampled-data system, where control signals (in $ \Sigma_R $) are piecewise-constant over intervals of length $ h\in\R^+ $, i.e.
\[\mathcal{U}_h=\{\upsilon\in\mathcal{U} \st \upsilon(t)=\upsilon(ih), t\in[ih,(i+1)h),i\in\N_0\}.\] 
The metric systems associated with the sampled-data retarded jump-diffusion systems can be defined as $ S_h(\Sigma_R)=(X_h,X_{h0},U_h,\underset{h}{\longrightarrow},Y_h,H_h) $, where $ X_h=X $, $ X_{h0}=X_0 $, $ U_h=\mathcal{U}_h $, $ Y_h=Y $, $ H_h=H $, and $\zeta_h\underset{h}{\overset{\upsilon_h}{\longrightarrow}} \zeta_h' $ if $ \zeta_h $ and $ \zeta_h' $ are measurable in $ \mathcal{F}_{ih} $ and $ \mathcal{F}_{(i+1)h} $, respectively, for some $ i\in\N_0 $, and there exists a solution $\xi_{t}\in  \mathbf{L}_{\mathcal{F}_{t}}^q([-\tau,0];\R^n)$ of $ \Sigma_R $ satisfying $ \xi_{t}=\zeta_h $ and $ \xi_{h,\zeta_h,\upsilon_h}=\zeta_h' $ $ \PP $-a.s. In other words, a finite state-run of $ S_h(\Sigma_R) $, represented by $ \zeta_0\underset{h}{\overset{\upsilon_0}{\longrightarrow}} \zeta_1\underset{h}{\overset{\upsilon_1}{\longrightarrow}} \cdots \underset{h}{\overset{\upsilon_{k-1}}{\longrightarrow}} \zeta_k $, where $ \upsilon_i\in U_h $ and $ \zeta_{i+1}=\xi_{h,\zeta_i,\upsilon_i} $ $ \PP $-a.s. for $ i\in\{0,1,\ldots,k-1\} $, captures solutions of RJDS $ \Sigma_R $ at the sampling times $ t=0,h,\ldots ,kh $, started from $ \zeta_0\in X_0 $ and resulting from control input $ \upsilon $ obtained by the concatenation of input signals $ \upsilon_i \in \mathcal{U}_h$. Moreover, the corresponding finite output-run is $ \{y_0, y_1,\ldots,y_{k}\} $. Similarly, we consider metric systems corresponding to non-probabilistic sampled-data retarded systems denoted by $S_h(\ol\Sigma_R)=(\ol X_h,\ol X_{h0},\ol U_h,\underset{h}{\longrightarrow},\ol Y_h,\ol H_h) $ where $ \ol X_h= \mathcal{C}([-\tau, 0];\R^n) $, $ \ol X_{h0}\subseteq \ol X_h$, $ \ol U_h=\mathcal{U}_h $, $ \ol Y_h=\ol X_h$, $ \ol H_h(\zeta)=\zeta$, and $\zeta_h\underset{h}{\overset{\upsilon_h}{\longrightarrow}} \zeta_h' $ if $ \zeta_h'=\ol\xi_{h,\zeta_h,\upsilon_h}$. For later use, we represent an $\R^n$-valued output at the $k^{th}$ sampling instance starting from initial state $\zeta$ under input signal $\upsilon_h$ by $\ol\xi_{\zeta,\upsilon_h}(kh)$.\\
\section{Symbolic Models for RJDS}
\subsection{Finite Dimensional Abstractions}
In this subsection, we introduce a \textit{finite dimensional abstraction} for $ S_h(\ol \Sigma_R) $. Consider metric systems associated with the sampled-data retarded systems $ S_h(\ol\Sigma_R) $ and consider triple $ \rho=(h,N,\zeta_s) $ of parameters, where $ h\in\R^+ $ is the sampling time, $ N\in\N $ is a temporal horizon, and $ \zeta_s\in \mathcal{C}([-\tau, 0];\R^n)  $ is a source state. Let us define a metric system as
\begin{align*}
S_\rho(\ol\Sigma_R)&=(X_\rho,X_{\rho 0}, U_\rho,\underset{\rho}{\longrightarrow},Y_\rho,H_\rho),
\end{align*}
where
\begin{compactitem}
	\item $ X_\rho=\mathsf{U} $, $  X_{\rho 0}= X_\rho $, $  U_\rho=\mathsf{U}$, $ Y_\rho=Y_h $; 
	\item $x_\rho\underset{\rho}{\overset{u_\rho}{\longrightarrow}} x_\rho' $, where $ x_\rho=(u_1, u_2,\ldots ,u_N)\in X_\rho $, if and only if $ x'_\rho=(u_2,\ldots ,u_N,u_\rho) $;
	\item $ H_\rho(x_\rho)=\ol\xi_{Nh,\zeta_s,x_\rho}$.
\end{compactitem}
Here, we abuse notation by identifying $ x_\rho=(u_1, u_2,\ldots$,  $u_N)\in [\mathsf{U}]_\eta ^N $ as an input curve $\upsilon:[0,Nh)\to[\mathsf{U}]_\eta $ such that $\upsilon(t)= u_k$ for any $t\in[(k-1)h,kh)$ for $k\in\{1,\ldots,N\}$ in $\ol\xi_{Nh,\zeta_s,x_\rho}$. We use similar notations in the rest of the paper as well. Notice that the system $ S_\rho(\ol\Sigma_R) $ is deterministic, non-blocking, and \emph{finite dimensional} (but not necessarily symbolic unless $\mathsf{U}$ is a finite set). Note that $ H_\rho $ is the output map from non-probabilistic state $ x_\rho\in X_\rho $ to a $\mathcal{C}([-\tau, 0];\R^n)$-valued solution process $ \ol\xi_{Nh,\zeta_s,x_\rho} $ and corresponding $\R^n$-valued solution is represented by $ \ol\xi_{\zeta_s,x_\rho}(Nh)$.\\

The next theorem provides the results on the construction of finite dimensional abstractions which are approximately bisimilar to $S_h(\ol\Sigma_R)$.
\begin{theorem}\label{theorem2}
	Consider a retarded system $ \ol\Sigma_R $ corresponding to $\delta$-ISS-M$_q$ RJDS $\Sigma_R$ for $q\geq1$. Given any $ \varepsilon>0 $, let the sampling time $ h$, temporal horizon $ N $, and source  state $ \zeta_s $ be such that
	\begin{align}\label{asdf}
	\tilde\beta(\varepsilon,h))+\tilde{\beta}(\mathcal{Z}(\zeta_s),Nh)\leq\varepsilon,
	\end{align}
	where $ \mathcal{Z}(\zeta_s)=\underset{u_1\in\mathsf{U}}{\sup} \|\ol\xi_{h,\zeta_s,u_1}-\zeta_s\|^q_{[-\tau,0]}$. Then, the relation 
	\[{\mathcal{R}}_1=\{(\zeta,x_\rho)\in \ol X_h\times X_\rho \st \|\ol H_h(\zeta)-{H}_\rho(x_\rho)\|^q_{[-\tau,0]} \leq\varepsilon\}\]
	is an $ \varepsilon $-approximate bisimulation relation between $ S_h(\ol\Sigma_R) $ and $ {S}_\rho(\ol\Sigma_R) $.
\end{theorem}
\begin{proof}
	Consider any $ (\zeta,x_\rho)\in{\mathcal{R}}_1 $, where $ \zeta\in \ol X_h $ and $  x_\rho=(u_1,u_2, \ldots ,u_N)\in X_\rho$. Then we have $ \|\ol H_h(\zeta)-{H}_\rho(x_\rho)\|^q_{[-\tau,0]}\leq\varepsilon $. Thus condition (i) in Definition \ref{definition7} holds. Now we show that condition (ii) in Definition \ref{definition7} holds. Consider any $  \upsilon_h:[0,h[\to u_h$ for some $u_h\in\mathsf{U}$ and $ \zeta'=\ol\xi_{h,\zeta,\upsilon_h}  $. Consider $u_\rho=u_h$ and $ x_\rho'=(u_2,\ldots ,u_N,u_\rho) $ and let $ \ol{x}_\rho=(u_1,u_2,\ldots ,u_N,u_\rho)  $ denote input sequence in $ \mathsf{U}^{N+1} $. With the help of triangle inequality and $ \eqref{mmm1} $ one obtains the following chains of inequalities: \\	
	\begin{align}
	&\|\ol H_h(\zeta')-{H}_\rho(x_\rho')\|^q_{[-\tau,0]}=\|\ol H_h(\zeta')-H_\rho(\ol{x}_\rho)+H_\rho(\ol{x}_\rho)-{H}_\rho(x_\rho')\|^q_{[-\tau,0]}\nonumber\\
	&= \|\ol\xi_{h,\zeta,\upsilon_h}-\ol\xi_{(N+1)h,\zeta_s,\ol{x}_\rho}+\ol\xi_{(N+1)h,\zeta_s,\ol{x}_\rho}-\ol\xi_{Nh,\zeta_s,x_\rho'} \|^q_{[-\tau,0]}\nonumber\\
	&\leq\|\ol\xi_{h,\zeta,\upsilon_h}-\ol\xi_{(N+1)h,\zeta_s,\ol{x}_\rho}\|^q_{[-\tau,0]}+\|\ol\xi_{(N+1)h,\zeta_s,\ol{x}_\rho}-\ol\xi_{Nh,\zeta_s,x_\rho'} \|^q_{[-\tau,0]}\nonumber\nonumber\\
	&\leq \|\ol\xi_{h,\zeta,\upsilon_h}-\ol\xi_{h,\ol\xi_{Nh,\zeta_s,x_\rho},u_\rho}\|^q_{[-\tau,0]}+\|\ol\xi_{Nh,\ol\xi_{h,\zeta_s,u_1},x'_\rho}-\ol\xi_{Nh,\zeta_s,x_\rho'} \|^q_{[-\tau,0]}\nonumber\\
	&\leq\tilde{\beta}(\|\zeta-\ol\xi_{Nh,\zeta_s,x_\rho}\|^q_{[-\tau,0]},h)+\tilde{\beta}(\|\ol\xi_{h,\zeta_s,\upsilon_1}-\zeta_s\|^q_{[-\tau,0]},Nh)\leq\tilde{\beta}(\varepsilon,h)+\tilde{\beta}(\mathcal{Z}(\zeta_s),Nh)\leq\varepsilon.\nonumber
	\end{align}
	Hence, $ (\zeta',x'_\rho)\in{\mathcal{R}}_1 $. Thus condition (ii) in Definition \ref{definition7} holds. In a similar way, one can show that condition (iii) in Definition \ref{definition7} holds which completes the proof.
\end{proof}
Note that in the above theorem, given any $ \varepsilon>0 $, one can select temporal horizon $ N $ to be sufficiently large to enforce term $ \tilde{\beta}(\mathcal{Z}(\zeta_s),Nh) $ to be sufficiently small. This results in $ \tilde{\beta}(\varepsilon,h)<\varepsilon $ which enforces a lower bound for the sampling time $ h $.
Now we establish the results on the existence of finite dimensional abstraction $S_\rho(\ol\Sigma_R)$ such that $ S_h(\ol\Sigma_R) \cong_\mathcal{S}^\varepsilon  S_\rho(\ol\Sigma_R) $ given the result in Theorem \ref{theorem2}.
\begin{theorem}\label{theorem6}
	Consider the results in Theorem \ref{theorem2}. If we select
	\[ X_{h0}\hspace{-.2em}\subseteq\hspace{-.2em}\{\zeta\in \mathcal{C}([-\tau, 0];\R^n)\hspace{-.2em} \st \hspace{-.2em}\|\ol H_h(\zeta)\hspace{-.2em}-\hspace{-.2em}H_{\rho}(x_{\rho 0}\hspace{-.2em})\|^q_{[-\tau,0]}\hspace{-.2em}\leq\hspace{-.2em}\varepsilon, \exists x_{\rho 0}\hspace{-.2em}\in\hspace{-.2em} X_{\rho 0}\},\] 
	\normalsize
	then we have $ S_h(\ol\Sigma_R) \cong_\mathcal{S}^\varepsilon S_\rho(\ol\Sigma_R) $.
\end{theorem}
\begin{proof}
	For every $ \zeta\in \ol X_{h0}$, there always exists $ x_{\rho 0}\in X_{\rho 0} $ such that  $ \|\ol H_h(\zeta)- H_{\rho}(x_{\rho 0})\|^q_{[-\tau,0]}\leq\varepsilon $. Hence $ (\zeta,x_{\rho 0})\in{\mathcal{R}}_1 $.  In a similar way, we can show that for every $ x_{\rho 0}\in X_{\rho 0} $ there exists $ \zeta\in \ol X_{h0}$ such that $ (\zeta,x_{\rho 0})\in {\mathcal{R}}_1$, which completes the proof.
\end{proof}

\subsection{Finite Abstractions}
In this subsection, we provide a finite (a.k.a. symbolic) abstraction for $ S_h(\ol\Sigma_R) $ by quantizing input set $ \mathsf{U}$. Let us consider tuple $ \ol\rho=(h,N,\zeta_s,\eta) $, where $ \eta>0 $ is a quantization parameter and the quantized input set is denoted by $ [\mathsf{U}]_\eta $ (cf. Notation in Subsection \ref{II1}). Now, we can define the corresponding finite systems as 
\[S_{\ol\rho}(\ol\Sigma_R) =(X_{\ol\rho}, X_{\ol\rho 0}, U_{\ol\rho}, \underset{\ol\rho}{\longrightarrow}, Y_{\ol\rho,} H_{\ol\rho}),\]
where
\begin{compactitem}
	\item $ X_{\ol\rho}=[\mathsf{U}]_\eta^N $,$  X_{\ol\rho 0}= X_{\ol\rho} $, $  U_{\ol\rho}=[\mathsf{U}]_\eta $, $ Y_{\ol\rho}=Y_h $; 
	\item $x_{\ol\rho}\underset{{\ol\rho}}{\overset{\ol u_{\ol\rho}}{\longrightarrow}} x_{\ol\rho}' $, where $ x_{\ol\rho}=(\ol u_1, \ol u_2,\ldots ,\ol u_N)\in X_{\ol\rho} $, if and only if $ x'_{\ol\rho}=(\ol u_2,\ldots ,\ol u_N,\ol u_{\ol\rho}) $;
	\item $ H_{\ol\rho}(x_{\ol\rho})=\ol\xi_{Nh,\zeta_s,x_{\ol\rho}} $.
\end{compactitem}
A finite state-run of $ S_{\ol\rho}(\ol\Sigma_R) $ is represented by $ x_{\ol\rho}(0)\underset{\ol\rho}{\overset{\upsilon_0}{\longrightarrow}} x_{\ol\rho}(1)\underset{\ol\rho}{\overset{\upsilon_1}{\longrightarrow}} \cdots \underset{\ol\rho}{\overset{\upsilon_{k-1}}{\longrightarrow}} x_{\ol\rho}(k) $, where $\upsilon_i\in U_{\ol\rho}$. 
For later use, we denote $\tilde Y_{y_0,\upsilon_{\ol\rho}}:\N_0\ra\R^n$ an $\R^n$-valued output-run of $S_{\ol\rho}(\ol\Sigma_R)$ starting from $y_0=\ol\xi_{\zeta_s,x_{\ol\rho}(0)}(Nh)$ under input signal $\upsilon_{\ol\rho}$.
In order to provide an approximate bisimulation relation between sampled retarded systems and symbolic models, we need the following technical lemma.   
\begin{lemma}\label{lem1}
	Consider a retarded system $ \ol\Sigma_R $ corresponding to $\delta$-ISS-M$_q$ RJDS $\Sigma_R$ for $q\geq1$ and a quantization parameter $\eta$ such that $ 0<\eta\leq span(\mathsf{U}) $. Then the relation ${\mathcal{R}}_2$ given by
	\begin{align}
	{\mathcal{R}}_2=\{&(x_\rho,x_{\ol\rho})\in X_\rho\times X_{\ol\rho} \st x_\rho=(u_1,u_2,\ldots ,u_N), x_{\ol\rho}=([u_1]_\eta,[u_2]_\eta,\ldots ,[u_N]_\eta) \}
	\end{align}
	is a $ \gamma(\eta) $-approximate bisimulation relation between $ S_\rho(\ol\Sigma_R)  $ and $ S_{\ol\rho}(\ol\Sigma_R)  $, and $ S_\rho(\ol\Sigma_R) \cong_\mathcal{S}^{\gamma(\eta)} S_{\ol\rho}(\ol\Sigma_R) $.
\end{lemma}
\begin{proof}
	Let $ (x_\rho,x_{\ol\rho})\in{\mathcal{R}}_2 $, then $ \|u_i-[u_i]_\eta\|\leq \eta $ for $ i\in\{1,2,\ldots,N \}$ implies that $ \|x_\rho-x_{\ol\rho}\|_\infty\leq \eta $. By using \eqref{mmm1}, one obtains
	\begin{align}
	\|{H}_\rho(x_\rho)\hspace{-.2em}-\hspace{-.2em}{H}_{\ol\rho}(x_{\ol\rho})\|^q\hspace{-.2em}&=\hspace{-.2em}\|\ol\xi_{Nh,\zeta_s,x_\rho}\hspace{-.2em}-\hspace{-.2em}\ol\xi_{Nh,\zeta_s,x_{\ol\rho}}\|^q_{[-\tau,0]}\nonumber\leq\gamma(\|x_\rho-x_{\ol\rho}\|_\infty)\leq\gamma(\eta).\nonumber
	\end{align}
	Then, the first condition in Definition \ref{definition7} holds. Now, consider any $ (x_\rho,x_{\ol\rho}) \in{\mathcal{R}}_2 $, where $x_\rho=(u_1,u_2,\ldots ,u_N)$ and $x_{\ol\rho}=(\ol u_1,\ol u_2,\ldots ,\ol u_N)$. Let $u\in U_\rho$ and consider $x_{\rho}\underset{{\rho}}{\overset{u}{\longrightarrow}} x_{\rho}' := (u_2,\ldots ,u_N,u)$ in $S_\rho(\ol\Sigma_R)$. Choose $\ol u=[u]_\eta$ and consider $x_{\ol\rho}\underset{{\ol\rho}}{\overset{\ol u}{\longrightarrow}} x_{\ol\rho}' := (\ol u_2,\ldots ,\ol u_N,\ol u)$ in $ S_{\ol\rho}(\ol\Sigma_R)$. It is obvious that $ (x_\rho',x_{\ol\rho}') \in{\mathcal{R}}_2 $ and, hence, condition (ii) in Definition \ref{definition7} holds. Similarly, condition (iii) in Definition \ref{definition7} holds which shows $ {\mathcal{R}}_2 $ is a $ \gamma(\eta)$-approximate bisimulation relation between $ S_\rho(\ol\Sigma_R)  $ and $  S_{\ol\rho}(\ol\Sigma_R) $. For any $x_{\rho 0}:=(u_1,u_2,\ldots ,u_N)\in X_{\rho 0}$, there always exists $x_{\ol\rho 0}:=([u_1]_\eta,[u_2]_\eta,\ldots ,[u_N]_\eta)\in X_{\ol\rho 0}$ and, hence,  $(x_{\rho 0},x_{\ol\rho 0})\in{\mathcal{R}}_2 $. Note that the existence of such $x_{\ol\rho 0}$ is guaranteed by $\mathsf{U}$ being a finite union of boxes and by the inequality $\eta\leq span(\mathsf{U})$. Moreover, for any $x_{\ol\rho 0}\in X_{\ol\rho 0}$ and by choosing $x_{\rho 0}=x_{\ol\rho 0}$, one readily gets $(x_{\rho 0},x_{\ol\rho 0})\in{\mathcal{R}}_2$ and, hence, $  S_\rho(\ol\Sigma_R) \cong_\mathcal{S}^{\gamma(\eta)} S_{\ol\rho}(\ol\Sigma_R) $.
\end{proof}
Now we provide the main results of this subsection on establishing an approximate bisimulation relation between $ S_h(\ol\Sigma_R) $ and $ S_{\ol\rho}(\ol\Sigma_R) $, which is an immediate consequence of the transitivity property of approximate bisimulation relations \cite[Proposition 4]{Girard1} as recalled next.
\begin{proposition}\label{propo}
	Consider metric systems $ S_1 $, $ S_2 $, and $ S_3 $ such that $ S_1\cong_\mathcal{S}^\delta S_2 $ and $ S_2\cong_\mathcal{S}^{\h \delta} S_3 $, for some $ \delta,\h \delta\in\R^+_0 $. Then, we have  $ S_1\cong_\mathcal{S}^{\delta+\h\delta} S_3 $.
\end{proposition}
Now we provide the first main result of this section.
\begin{theorem}\label{theorem3}
	Consider a retarded system $ \ol\Sigma_R $ corresponding to $\delta$-ISS-M$_q$ RJDS $\Sigma_R$ for $q\geq1$. Given any $\varepsilon>0$ and a quantization parameter $ 0<\eta\leq span(\mathsf{U}) $, consider the results in Theorem \ref{theorem2} and Lemma \ref{lem1}. Then, the relation $ {\mathcal{R}} $ given by
	\begin{align*}
	{\mathcal{R}}=\big\{&(x_h,x_{\ol\rho})\in X_h\times \ol X_{\ol\rho} \st \exists x_\rho\in X_\rho, (x_h,x_\rho)\in{\mathcal{R}}_1 \ \text{and} \ (x_\rho,x_{\ol\rho})\in{\mathcal{R}}_2 \big\} 
	\end{align*}
	is an $ (\varepsilon+\gamma(\eta)) $-approximate bisimulation relation between $ S_h(\ol\Sigma_R) $ and $ S_{\ol\rho}(\ol\Sigma_R) $.
\end{theorem}
Note that relations ${\mathcal{R}}_1$ and ${\mathcal{R}}_2$ in Theorem \ref{theorem3} have been defined in Theorem \ref{theorem2} and Lemma \ref{lem1}, respectively. \\

Having the result in Theorem \ref{theorem3}, now we provide one of the main results of the paper that quantifies the closeness of the output trajectories of sampled retarded jump-diffusion systems $S_h(\Sigma_R)$ and those of symbolic models $S_{\ol\rho}(\ol\Sigma_R)$.
In order to prove this result, we raise a supplementary assumption on $\delta$-ISS-$M_q$ Lyapunov functions $V$.
\begin{assumption}\label{assum1}
	For a $\delta$-ISS-$M_q$ Lyapunov function $V$ as in Definition \ref{definition4}, following conditions hold:
	\begin{compactitem}
		\item[(i)] for function $\omega$ in Definition~\ref{definition4} there exists $\kappa\in\R^+$ such that $\omega(x,\h x)\geq \kappa V(x,\h x)$ for all $ x,\h x\in\R^n$; 
		\item[(ii)] there exists some constant $\alpha\in\R_0^+$ such that
		\begin{align}\label{assum_eq}
		&\frac{1}{2}\mathsf{Tr}(g(\zeta)g^T(\h\zeta)\partial_{\zeta(0),\zeta(0)}V(\zeta(0),\h\zeta(0)))-\frac{1}{2}\mathsf{Tr}\bigg(\begin{bmatrix}
		g(\zeta)\\ 
		g(\h\zeta)
		\end{bmatrix} \begin{bmatrix}
		g^T(\zeta)& g^T(\h\zeta)
		\end{bmatrix}\mathcal{H}(V)(\zeta(0),\h\zeta(0))\bigg)\nonumber\\
		&+\sum_{i=1}^{\tilde{r}}\lambda_i\Big(V\big(\zeta(0)+r(\zeta)e_i,\h\zeta(0)\big)- V\big(\zeta(0)+r(\zeta)e_i,\h\zeta(0)+r(\h\zeta)e_i\big)\Big)\leq \alpha
		\end{align}
		for all $\zeta,\h\zeta\in\mathcal{C}([-\tau,0];\R^n)$, where $ \mathcal{H}(V)(x,\h x) $ denotes the Hessian matrix of $ V $ at $x,\h x\in\R^{n} $. 
	\end{compactitem}
\end{assumption} 
Note that condition (ii) in Assumption~\ref{assum1} is not restrictive, provided that $V$ is restricted to a compact subset of $\R^n \times \R^n$, the Hessian matrix $\mathcal{H}(V)(x,\h x)$ of $V$ is a positive semidefinite matrix in $\R^{2n\times2n}$, and $\partial_{x,x}V(x,\h x)\leq\ol P$ for some $\ol P$ a positive semidefinite matrix in $\R^{n\times n}$. With these conditions and Lipschitz assumptions on $g(\cdot)$ and $r(\cdot)$, one can always find $\alpha\in\R_0^+$ satisfying \eqref{assum_eq}.
\begin{theorem}\label{mainthm1}
	Consider an RJDS $\Sigma_R$ admitting a $\delta$-ISS-$M_q$ Lyapunov function $V$ for $q\geq1$ as in Definition \ref{definition4} and satisfying conditions in Assumption \ref{assum1}. For any abstraction parameters $ \ol\rho=(h,N,\zeta_s,\eta) $ and any $\varepsilon\in\R^+$ satisfying \eqref{asdf}, consider finite abstraction $S_{\ol\rho}(\ol\Sigma_R)$ and results in Theorem \ref{theorem3}. For any $\R^n$-valued output-run $\xi_{\zeta,\upsilon_h}$ of $S_h(\Sigma_R)$, there exists an $\R^n$-valued output-run $\tilde Y_{\ol\xi_{\zeta_s,x_{\ol\rho}(0)}(Nh),\upsilon_{\ol\rho}}$ of $S_{\ol\rho}(\ol\Sigma_R)$, and vice versa, such that following inequalities hold
	\begin{align}\label{rrr1}
	\mathbb{P}\left\{\sup_{0\leq k\leq T_d}\| \xi_{\zeta,\upsilon_h}(kh)-\tilde Y_{\ol\xi_{\zeta_s,x_{\ol\rho}(0)}(Nh),\upsilon_{\ol\rho}}(k)\|^q\geq\epsilon+\varepsilon+\gamma(\eta)\mid \zeta\right\} \leq 1-\mathsf{e}^{-\frac{\alpha T_d h}{\underline{\alpha}(\epsilon)}},\quad \text{if } \underline{\alpha}(\epsilon)\geq\frac{\alpha}{\kappa},
	\end{align}
	\begin{align}\label{rrr2}
	\mathbb{P}\left\{\sup_{0\leq k\leq T_d}\| \xi_{\zeta,\upsilon_h}(kh)-\tilde Y_{\ol\xi_{\zeta_s,x_{\ol\rho}(0)}(Nh),\upsilon_{\ol\rho}}(k)\|^q\geq\epsilon+\varepsilon+\gamma(\eta)\mid \zeta\right\} \leq \frac{(\mathsf{e}^{\kappa T_d h}-1)\alpha}{\kappa \underline{\alpha}(\epsilon)\mathsf{e}^{\kappa T_d h }},\quad \text{if } \underline{\alpha}(\epsilon)\leq\frac{\alpha}{\kappa}.
	\end{align}
\end{theorem}
\begin{proof}
	From the result in Theorem \ref{theorem3}, we have $  S_h(\ol\Sigma_R) \cong_\mathcal{S}^{\varepsilon+\gamma(\eta)} S_{\ol\rho}(\ol\Sigma_R) $ which implies $\|\ol\xi_{h,\zeta,\upsilon_h}-\ol\xi_{Nh,\zeta_s,x_{\ol\rho}}\|^q_{[-\tau,0]}\leq\varepsilon+\gamma(\eta)$ and consequently 
	\begin{align}
	\sup_{k\in\N_0}\|\ol\xi_{\zeta,\upsilon_h}(kh)-\tilde Y_{\ol\xi_{\zeta_s,x_{\ol\rho}(0)}(Nh),\upsilon_{\ol\rho}}(k)\|^q\leq\varepsilon+\gamma(\eta).
	\end{align}
	Hence, one obtains the following chain of (in)equalities:
	\begin{align}\label{ineq1}
	&\mathbb{P}\left\{\sup_{0\leq k\leq T_d}\| \xi_{\zeta,\upsilon_h}(kh)-\tilde Y_{\ol\xi_{\zeta_s,x_{\ol\rho}(0)}(Nh),\upsilon_{\ol\rho}}(k)\|^q\geq\epsilon+\varepsilon+\gamma(\eta)\mid \zeta\right\}\nonumber\\
	&=\mathbb{P}\left\{\sup_{0\leq k\leq T_d}\| \xi_{\zeta,\upsilon_h}(kh)-\ol\xi_{\zeta,\upsilon_h}(kh)+\ol\xi_{\zeta,\upsilon_h}(kh)-\tilde Y_{\ol\xi_{\zeta_s,x_{\ol\rho}(0)}(Nh),\upsilon_{\ol\rho}}(k)\|^q\geq\epsilon+\varepsilon+\gamma(\eta)\mid \zeta\right\}\nonumber\\
	&\leq\mathbb{P}\left\{\sup_{0\leq k\leq T_d}\{\| \xi_{\zeta,\upsilon_h}(kh)-\ol\xi_{\zeta,\upsilon_h}(kh)\|^q+\|\ol\xi_{\zeta,\upsilon_h}(kh)-\tilde Y_{\ol\xi_{\zeta_s,x_{\ol\rho}(0)}(Nh),\upsilon_{\ol\rho}}(k)\|^q\}\geq\epsilon+\varepsilon+\gamma(\eta)\mid \zeta\right\}\nonumber\\
	&\leq\mathbb{P}\left\{\sup_{0\leq k\leq T_d}\underline{\alpha}(\| \xi_{\zeta,\upsilon_h}(kh)-\ol\xi_{\zeta,\upsilon_h}(kh)\|^q)\geq\underline{\alpha}(\epsilon)\mid \zeta\right\}\nonumber\\
	&\leq \mathbb{P}\left\{\sup_{0\leq k\leq T_d}V(\xi_{\zeta,\upsilon_h}(kh),\ol\xi_{\zeta,\upsilon_h}(kh))\geq\underline{\alpha}(\epsilon)\mid \zeta \right\}.\nonumber\\
	\end{align}
	From the properties of $V$ in Definition \ref{definition4} and Assumption \ref{assum1}, one has
		\begin{align}\label{ineq2}
		\mathcal{L}V&(\xi_{t,\zeta,\upsilon},\overline{\xi}_{t,\zeta,\upsilon})
			=\hspace{-.4em}\begin{bmatrix}\partial_x V & \hspace{-.4em}\partial_{\hat{x}} V\end{bmatrix}\hspace{-.4em}\begin{bmatrix}
				f(\xi_{t,\zeta,\upsilon},\upsilon(t))\nonumber\\ 
				f(\ol\xi_{t,\zeta,\upsilon},\upsilon(t))
			\end{bmatrix}\hspace{-.3em}+\hspace{-.2em}\frac{1}{2}\mathsf{Tr}(g(\xi_{t,\zeta,\upsilon})g^T\hspace{-.2em}(\xi_{t,\zeta,\upsilon})\partial_{x,x}V)\nonumber\\
			&+\hspace{-.2em}\sum_{i=1}^{\tilde{r}}\hspace{-.2em}\lambda_i\Big(V(\xi_{\zeta,\upsilon}(t)\hspace{-.2em}+\hspace{-.2em}r(\xi_{t,\zeta,\upsilon})e_i,\ol\xi_{\zeta,\upsilon}(t))\hspace{-.2em}-\hspace{-.2em}V(\xi_{\zeta,\upsilon}(t),\ol\xi_{\zeta,\upsilon}(t))\Big)\hspace{-.2em}\nonumber\\
			=&\begin{bmatrix}\partial_x V & \partial_{\hat{x}} V\end{bmatrix}\begin{bmatrix}
				f(\xi_{t,\zeta,\upsilon},\upsilon(t))\\ 
				f(\ol\xi_{t,\zeta,\upsilon},\upsilon(t))
			\end{bmatrix}+\frac{1}{2}\mathsf{Tr}\bigg(\begin{bmatrix}
				g(\xi_{t,\zeta,\upsilon})\\ 
				g(\ol\xi_{t,\zeta,\upsilon})
			\end{bmatrix} \begin{bmatrix}
			g^T(\xi_{t,\zeta,\upsilon}) & g^T(\ol\xi_{t,\zeta,\upsilon})
		\end{bmatrix}\mathcal{H}(V)\bigg)\nonumber\\
		&+\sum_{i=1}^{\tilde{r}}\hspace{-.2em}\lambda_i\Big(V(\xi_{\zeta,\upsilon}(t)\hspace{-.2em}+\hspace{-.2em}r(\xi_{t,\zeta,\upsilon})e_i,\ol\xi_{\zeta,\upsilon}(t)\hspace{-.2em}+\hspace{-.2em}r(\ol\xi_{t,\zeta,\upsilon})e_i)-V(\xi_{\zeta,\upsilon}(t),\ol\xi_{\zeta,\upsilon}(t))\Big)\hspace{-.2em}\nonumber\\
		&+\hspace{-.2em}\frac{1}{2}\mathsf{Tr}(g(\xi_{t,\zeta,\upsilon})g^T\hspace{-.2em}(\xi_{t,\zeta,\upsilon})\partial_{x,x}V)+\sum_{i=1}^{\tilde{r}}\lambda_iV(\xi_{\zeta,\upsilon}(t)+r(\xi_{t,\zeta,\upsilon})e_i,\ol\xi_{\zeta,\upsilon}(t))\nonumber\\
		&-\frac{1}{2}\mathsf{Tr}\bigg(\begin{bmatrix}
			g(\xi_{t,\zeta,\upsilon})\\ 
			g(\ol\xi_{t,\zeta,\upsilon})
		\end{bmatrix} \begin{bmatrix}
		g^T(\xi_{t,\zeta,\upsilon})& g^T(\ol\xi_{t,\zeta,\upsilon})
	\end{bmatrix}\mathcal{H}(V)\bigg)-\sum_{i=1}^{\tilde{r}}\lambda_i V(\xi_{\zeta,\upsilon}(t)+r(\xi_{s,\zeta,\upsilon})e_i,\ol\xi_{\zeta,\upsilon}(t)+r(\ol\xi_{t,\zeta,\upsilon})e_i)\nonumber\\
	\leq&-\kappa V(\xi_{\zeta,\upsilon}(t),\ol\xi_{\zeta,\upsilon}(t))+\alpha,\nonumber\\
	\end{align}
	for any $t\in\R_0^+$, $\upsilon_h\in\mathcal{U}_h$ and $\zeta\in\mathcal{C}^b_{\mathcal{F}_0}([-\tau, 0];\R^n)$. Using inequalities \eqref{ineq1}, \eqref{ineq2}, and the result in \cite[Theorem 1, pp. 79]{kushnerbook}, one obtains the relations \eqref{rrr1} and \eqref{rrr2}. In a similar way, we can prove the other direction. 
\end{proof}
Inequalities \eqref{rrr1} and \eqref{rrr2} lower bound the probability such that the distance between output trajectories of the finite abstraction and those of the corresponding sampled retarded jump-diffusion system remains close over finite time horizon. One can leverage the result in Theorem \ref{mainthm1} to synthesize control policies for finite abstractions and refine them to the original systems while providing guarantee on the probability of satisfaction. Similar relations were established in \cite{julius1}, \cite{zamani2014symbolic}, and \cite{Abolfazl2019} for stochastic hybrid systems, stochastic control systems, and interconnected stochastic control systems and their (in)finite abstractions. For a detailed discussion on how inequalities \eqref{rrr1} and \eqref{rrr2} can be used to provide a lower bound on the probability of satisfying the specification for the original systems, we kindly refer the interested readers to \cite[Section $6$]{Abolfazl2019}.

\section{An Example}
To show the effectiveness of the proposed results, we consider a simple thermal model of ten-room building as shown schematically in Figure \ref{ten_room}. The model used here is similar to that used in \cite{zamani2015symbolic}. In addition, we modified arrangement of the rooms to increase state-space dimensions and considered that the dynamic is affected by delays in states and by jumps (modeling door and window opening). The dynamic of the considered delayed jump-diffusion system $ \Sigma_D $ is given by the following delayed stochastic differential equations: 
\begin{figure}[tbp]
	\includegraphics[scale=0.7]{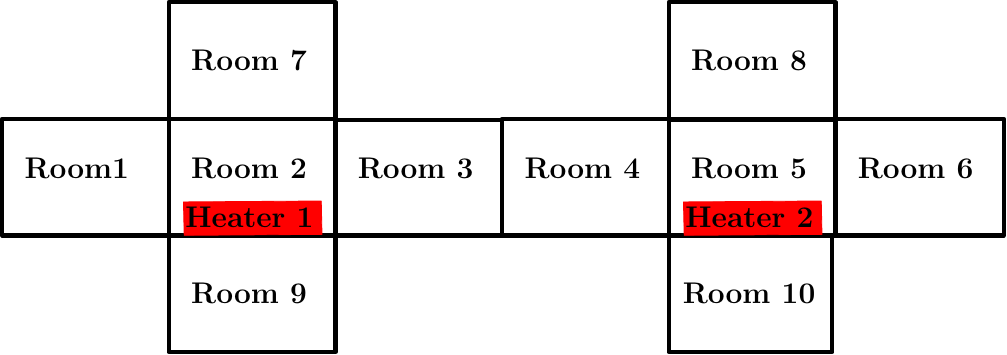}
	\caption{A schematic of ten-room building. }
	\label{ten_room}
\end{figure} 
\small
\begin{align*}
	\diff\xi_1(t)=&(\alpha_{21}(\xi_2(t)-\xi_1(t))+\alpha_{e_1}(T_e-\xi_1(t))-\alpha_{1\tau_1}\xi_1(t-\tau_1)) \diff t+(g_1\xi_1(t)+g_{1\tau_2}\xi_1(t-\tau_2))\diff W_t^1+r_1\xi_1(t) \diff P_t^1,\nonumber\\
	\diff\xi_2(t) =&(\alpha_{12}(\xi_1(t) - \xi_2(t)) + \alpha_{72}(\xi_7(t) - \xi_2(t)) + \alpha_{92}(\xi_9(t) - \xi_2(t)) + \alpha_{32}(\xi_3(t) - \xi_2(t)) + \alpha_{e_2}(T_e - \xi_2(t))\nonumber\\&-\alpha_{2\tau_1}\xi_2(t-\tau_1)+ \alpha_{H_1}(T_h - \xi_2(t))\upsilon_1) \diff t+(g_2\xi_2(t)+g_{2\tau_2}\xi_2(t-\tau_2))\diff W_t^2+r_1\xi_2(t) \diff P_t^2,\nonumber\\
	\diff\xi_3(t) =&(\alpha_{23}(\xi_2(t) - \xi_3(t))+\alpha_{43}(\xi_4(t) - \xi_3(t))+\alpha_{e_3}(T_e - \xi_3(t))-\alpha_{3\tau_1}\xi_3(t-\tau_1)) \diff t\nonumber\\&+(g_3\xi_3(t)+g_{3\tau_2}\xi_3(t-\tau_2))\diff W_t^3+r_3\xi_3(t) \diff P_t^3,\nonumber\\
	\diff\xi_4(t) =&(\alpha_{34}(\xi_3(t) - \xi_4(t))+\alpha_{54}(\xi_5(t) - \xi_4(t))+\alpha_{e_4}(T_e - \xi_4(t))-\alpha_{4\tau_1}\xi_4(t-\tau_1)) \diff t\nonumber\\&+(g_4\xi_4(t)+g_{4\tau_2}\xi_4(t-\tau_2))\diff W_t^4+r_4\xi_4(t) \diff P_t^4,\nonumber\\
	\diff\xi_5(t) =&(\alpha_{45}(\xi_4(t) - \xi_5(t)) + \alpha_{85}(\xi_8(t) - \xi_5(t)) + \alpha_{105}(\xi_{10}(t) - \xi_5(t)) + \alpha_{65}(\xi_6(t) - \xi_5(t)) + \alpha_{e_5}(T_e - \xi_5(t))\nonumber\\&-\alpha_{5\tau_1}\xi_5(t-\tau_1)+ \alpha_{H_2}(T_h  -\xi_5(t))\upsilon_2) \diff t+(g_5\xi_5(t)+g_{5\tau_2}\xi_5(t-\tau_2))\diff W_t^5+r_5\xi_5(t) \diff P_t^5,\nonumber\\
	\diff\xi_6(t) =&(\alpha_{56}(\xi_5(t) - \xi_6(t))+\alpha_{e_6}(T_e - \xi_6(t))-\alpha_{6\tau_1}\xi_6(t-\tau_1)) \diff t+(g_6\xi_6(t)+g_{6\tau_2}\xi_6(t-\tau_2))\diff W_t^6+r_6\xi_6(t) \diff P_t^6,\nonumber\\
	\diff\xi_7(t) =&(\alpha_{27}(\xi_2(t) - \xi_7(t))+\alpha_{e_7}(T_e - \xi_7(t))-\alpha_{7\tau_1}\xi_7(t-\tau_1)) \diff t+(g_7\xi_7(t)+g_{7\tau_2}\xi_7(t-\tau_2))\diff W_t^7+r_7\xi_7(t) \diff P_t^7,\nonumber\\
	\diff\xi_8(t) =&(\alpha_{58}(\xi_5(t) - \xi_8(t))+\alpha_{e_8}(T_e - \xi_8(t))-\alpha_{8\tau_1}\xi_8(t-\tau_1)) \diff t+(g_8\xi_8(t)+g_{8\tau_2}\xi_8(t-\tau_2))\diff W_t^8+r_8\xi_8(t) \diff P_t^8,\nonumber\\
	\diff\xi_9(t) =&(\alpha_{29}(\xi_2(t) - \xi_9(t))+\alpha_{e_9}(T_e - \xi_9(t))-\alpha_{9\tau_1}\xi_9(t-\tau_1)) \diff t+(g_9\xi_9(t)+g_{9\tau_2}\xi_9(t-\tau_2))\diff W_t^9+r_9\xi_9(t) \diff P_t^9,\nonumber\\
	\diff\xi_{10}(t) =&(\alpha_{510}(\xi_5(t) - \xi_{10}(t))+\alpha_{e_{10}}(T_e - \xi_{10}(t))-\alpha_{10\tau_1}\xi_{10}(t-\tau_1)) \diff t+(g_{10}\xi_{10}(t)+g_{10\tau_2}\xi_{10}(t-\tau_2))\diff W_t^{10}\nonumber\\&+r_{10}\xi_{10}(t) \diff P_t^{10},\nonumber
\end{align*}
\normalsize where the terms $ W^i_t $ and $ P^i_t $, $ i\in\{1,2,\ldots,10\} $, denote the standard Brownian motion and Poisson process with rate $ \lambda_i=0.1 $, respectively; $ \xi_i $, $ i\in\{1,2,\ldots,10\} $, denote the temperature in each room; $ T_e=15 $ (degree Celsius) is the external temperature; $ T_{H_1}=T_{H_2}=100 $ are the temperatures of two heaters; $ \tau_1=10 $ time units and $ \tau_2=5 $ time units are state delays in drift and diffusion terms, respectively; and the control inputs $ \upsilon_1 $, $ \upsilon_2 $ are amounted to $1$ if corresponding heaters are on and to $0$ if corresponding heaters are off. Here, we assume that at most one heater is on at each time instance which results in the finite input set $\mathsf{U}=\{(1,0),(0,1),(0,0)\}$. The system parameters are chosen as $ \alpha_{ij}=5\times10^{-2}$ for all $i\neq j$ and $i,j\in \{1,2,\ldots,10\}$, $ \alpha_{e2}=\alpha_{e5}=8\times10^{-3} $, $ \alpha_{e1}=\alpha_{e3}=\alpha_{e4}=\alpha_{e6}=\alpha_{e7}=\alpha_{e8}=\alpha_{e9}=\alpha_{e10}=5\times10^{-3} $, $ \alpha_{H_1}=\alpha_{H_2}=3.6\times10^{-3 }$, $ \alpha_{i\tau_1}=1\times10^{-4} $, $ g_{i}=2\times10^{-3} $, $ g_{i\tau_2}=1\times10^{-4} $, and $ r_i=1\times10^{-5} $, for $ i\in\{1,2,\ldots,10\} $. In this example, we work on the subset $\mathsf{D}=[17,22]^{10}$ of state space of $\Sigma_D$. 
\begin{figure}[tbp]
	\includegraphics[scale=0.6]{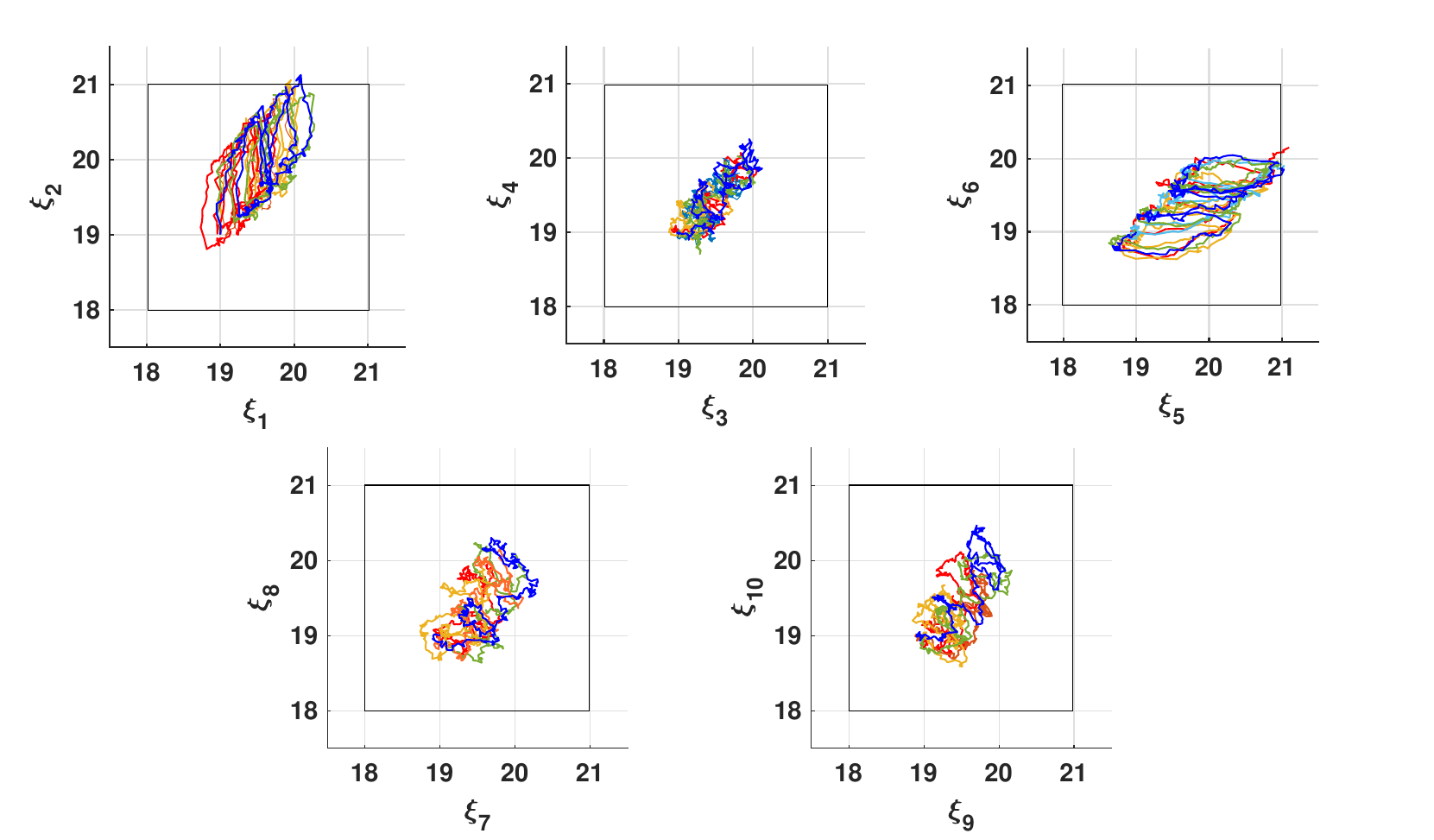}
	\caption{A few realizations of the solution process $ \xi_{\zeta,\upsilon} $ with initial condition $ \zeta\equiv[19, 19, 19, 19, 19, 19, 19, 19, 19, 19]^T$. }
	\label{safe_fig}
\end{figure} 
Using Lyapunov function $V(x,\h x)=(x-\h x)^TP(x-\h x)$, for all $x,\h x\in \mathsf{D}$, where 
$$ P\hspace{-.2em}=\hspace{-.2em}\begin{bmatrix}1.1126 &  -0.1205 &   0.0031 &   0.0004 &  -0.0001 &   0.0000  &  0.0036  &  0.0000 &   0.0036  &  0.0000\\
-0.1205 &   1.5002 &  -0.1232 &   0.0026 &   0.0006 &  -0.0001 &  -0.1205 &  -0.0001 &  -0.1205 &  -0.0001\\
0.0031 &  -0.1232 &   1.2308 &  -0.1182 &   0.0026  &  0.0004 &   0.0031  &  0.0004  &  0.0031  &  0.0004\\
0.0004  &  0.0026 &  -0.1182 &   1.2308 &  -0.1232  &  0.0031  &  0.0004  &  0.0031  &  0.0004  &  0.0031\\
-0.0001  &  0.0006 &   0.0026 &  -0.1232 &   1.5002 &  -0.1205 &  -0.0001 &  -0.1205 &  -0.0001 &  -0.1205\\
0.0000 &  -0.0001 &   0.0004  &  0.0031 &  -0.1205 &   1.1126 &   0.0000 &   0.0036  &  0.0000  &  0.0036\\
0.0036 &  -0.1205 &   0.0031  &  0.0004 &  -0.0001  &  0.0000 &   1.1126 &   0.0000  &  0.0036  &  0.0000\\
0.0000 &  -0.0001 &   0.0004  &  0.0031 &  -0.1205  &  0.0036  &  0.0000 &   1.1126  &  0.0000  &  0.0036\\
0.0036 &  -0.1205   & 0.0031  &  0.0004  & -0.0001 &   0.0000  &  0.0036 &   0.0000  &  1.1126  &  0.0000\\
0.0000  & -0.0001 &   0.0004  &  0.0031 &  -0.1205  &  0.0036 &   0.0000  &  0.0036 &   0.0000  &  1.1126
\end{bmatrix}\hspace{-.2em},
 $$ one can readily obtain the functions $\ul\alpha(s)=0.5029s$, $\ol\alpha(s)=0.8197s$, for $q=2$ satisfying conditions (i) and (ii) in Definition \ref{definition4}. Using the results in Lemma~\ref{lem3}, one obtains function $ \beta(s,t)=\e^{-\kappa t} s$, $ \forall s\in\R^+_0 $, with $ \kappa=0.6667 $. \\
By considering $ q=2 $, a constant function $\zeta_s\equiv[17, 17, 17, 17, 17, 17,17,17,17,17]^T\in \mathcal{C}([-\tau, 0];\R^n)$, where $\tau=10$, a precision $ \varepsilon=0.05 $ and by fixing sampling time $ h=15 $ time units, one can obtain temporal horizon $ N=9$ for $  S_\rho(\ol\Sigma_D) $, satisfying inequality \eqref{asdf} in Theorem~\ref{theorem2}. Therefore, the resulting cardinality of the set of states of $ S_\rho(\ol\Sigma_D) $ is $ 3^{9}= 19683$ and number of transitions are $3^{10}=59049$. The CPU time taken for computing finite abstraction with $N=9$ is accounted to 0.015. From inequalities \eqref{rrr1} and \eqref{rrr2}, one can observe that higher precision of the abstraction helps to improve the bound on the closeness of the trajectories. To get higher precision, one can increasing the value of $N$. For example, for $N=11$ one gets an abstraction with precision $\varepsilon=0.008$. However, it increases the size of the abstraction and the computation time which are $531441$ and $0.413$ seconds, respectively. Remark that since the input set is finite, the finite dimensional abstraction $ S_\rho(\ol\Sigma_D) $ is also symbolic. 

\begin{figure}[tbp]
	\includegraphics[scale=0.6]{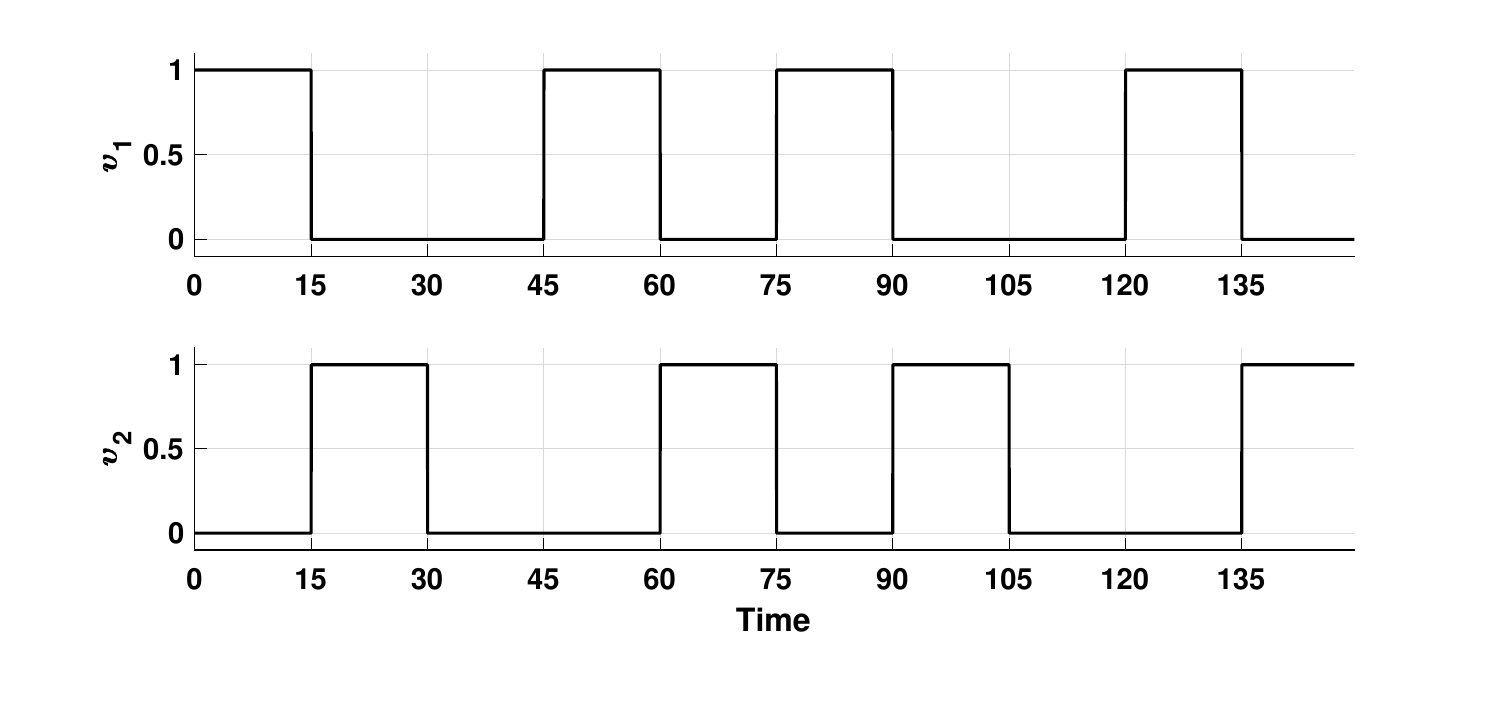}
	\caption{The evolution of input signals $ \upsilon_1 $ and $ \upsilon_2 $.}
	\label{input_fig}
\end{figure}  
\begin{figure}[tbp]
	\includegraphics[scale=0.6]{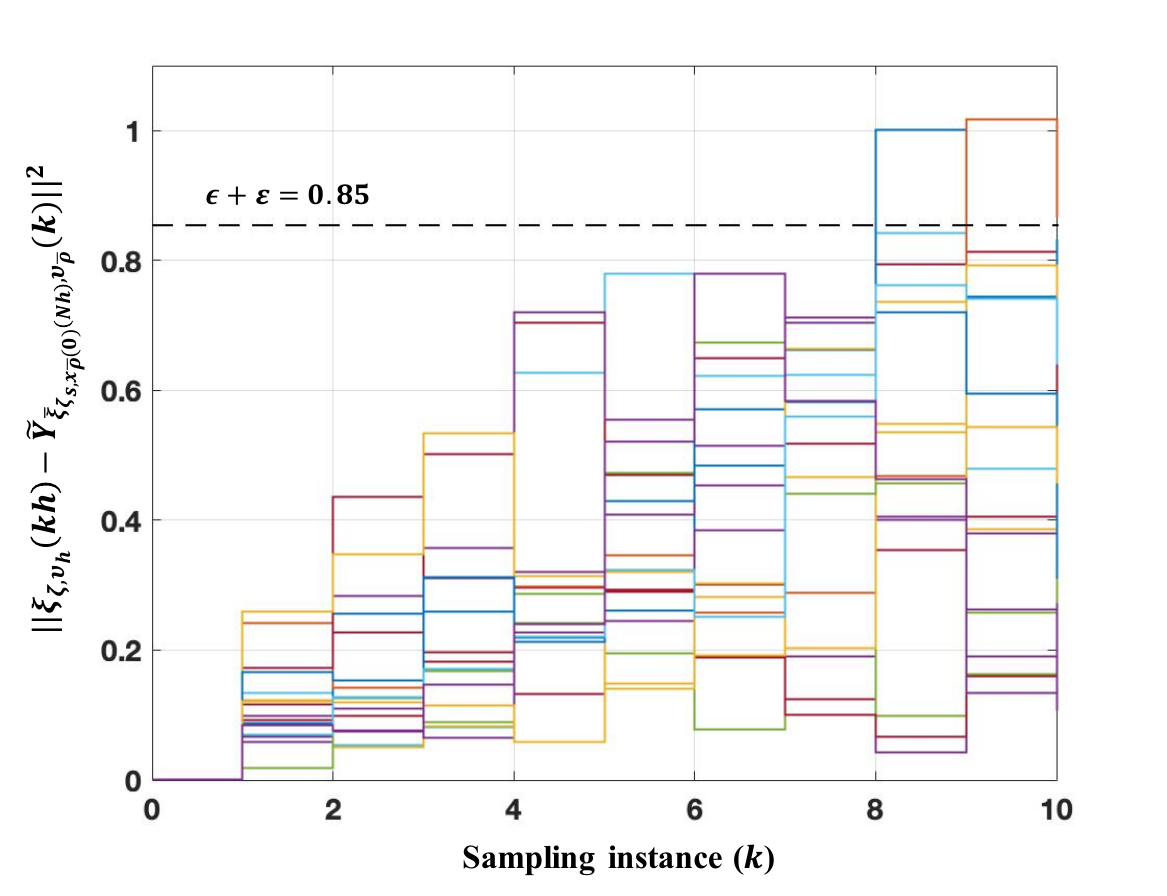}
	\caption{A few realizations of $ \| \xi_{\zeta,\upsilon_h}(kh)-\tilde Y_{\ol\xi_{\zeta_s,x_{\ol\rho}(0)}(Nh),\upsilon_{\ol\rho}}(k)\|^2$ for $T_d=10$ at sampling instances.}
	\label{closeness}
\end{figure}
Now consider the objective to design a controller enforcing the trajectories of $ \Sigma_D $ to stay within comfort zone $ \mathsf{W}=[18, 21]^{10} $. This corresponds to the LTL specification $\square \mathsf{W} $. The computation of the symbolic model $ S_{\rho}(\ol\Sigma_D) $ and the controller synthesis have been performed using tool {\tt QUEST} \cite{Jagtap2017} on a computer with CPU 3.5GHz Intel Core i7. The CPU time taken to synthesize controller for $N=9$ and $N=11$ is accounted to 2.07 and 2.56 seconds, respectively.\\
For the given dynamics and $\delta$-ISS-$M_2$ Lyapunov function $V$ on $\mathsf{D}$, we have
$\alpha=0.0012$ in Assumption~\ref{assum1}. For $\epsilon=0.8$, using the result in Theorem~\ref{mainthm1}, we guarantee that the the distance between output of $S_h(\Sigma_D)$ and that of $S_\rho(\ol\Sigma_D)$ will not exceed $\epsilon+\varepsilon=0.85$ over the discrete time horizon $\{0,15,\ldots,150\}$ with probability at least 75.8\%. To get a better lower bound for the aforementioned probability, one can reduce the time horizon or increase $\epsilon$. For example, if we consider discrete time horizon $\{0,15,\ldots,75\}$ and $\epsilon=1$, the lower bound on the probability will be 89.5\%. 
Figure~\ref{safe_fig} shows a few realizations of the closed-loop solution process $ \xi_{\zeta,\upsilon} $ starting form initial condition $ \zeta\equiv [19, 19, 19, 19, 19, 19, 19, 19, 19, 19]^T\in X_{h0} $. The synthesized control policies $ \upsilon_1 $ and $ \upsilon_2 $ are shown in Figure~\ref{input_fig}. The obtained probability (i.e., at least 75.8\%) is also empirically verified by computing distance between output trajectories of $S_h(\Sigma_R)$ and of $S_\rho(\ol\Sigma_R)$ (i.e., $ \| \xi_{\zeta,\upsilon_h}(kh)-\tilde Y_{\ol\xi_{\zeta_s,x_{\ol\rho}(0)}(Nh),\upsilon_{\ol\rho}}(k)\|^2$) using $5000$ runs. Several realizations are shown in Figure \ref{closeness}.
\bibliographystyle{alpha}
\bibliography{IEEEtran1}

\newcommand{\etalchar}[1]{$^{#1}$}
\begin{thebibliography}{ZMEM{\etalchar{+}}14}

\bibitem[BSP08]{bandyopadhyay2008deterministic}
M.~Bandyopadhyay, T.~Saha, and R.~Pal.
\newblock Deterministic and stochastic analysis of a delayed allelopathic
  phytoplankton model within fluctuating environment.
\newblock {\em Nonlinear Analysis: Hybrid systems}, 2(3):958--970, May 2008.

\bibitem[CGG13]{Girard3}
E.~Le Corronc, A.~Girard, and G.~Goessler.
\newblock Mode sequences as symbolic states in abstractions of incrementally
  stable switched systems.
\newblock In {\em 52nd IEEE Conference on Decision and Control}, pages
  3225--3230, December 2013.

\bibitem[GB94]{411398}
L.~El Ghaoui and V.~Balakrishnan.
\newblock Synthesis of fixed-structure controllers via numerical optimization.
\newblock In {\em Proceedings of 1994 33rd IEEE Conference on Decision and
  Control}, volume~3, pages 2678--2683, December 1994.

\bibitem[Gir14]{Girard2}
A.~Girard.
\newblock Approximately bisimilar abstractions of incrementally stable finite
  or infinite dimensional systems.
\newblock In {\em 53rd IEEE Conference on Decision and Control}, pages
  824--829, December 2014.

\bibitem[GP07]{Girard1}
A.~Girard and G.~J. Pappas.
\newblock Approximation metrics for discrete and continuous systems.
\newblock {\em IEEE Transactions on Automatic Control}, 52(5):782--798, May
  2007.

\bibitem[GPT10]{5342460}
A.~Girard, G.~Pola, and P.~Tabuada.
\newblock Approximately bisimilar symbolic models for incrementally stable
  switched systems.
\newblock {\em IEEE Transactions on Automatic Control}, 55(1):116--126, January
  2010.

\bibitem[HM09]{huang2009input}
L.~Huang and X.~Mao.
\newblock On input-to-state stability of stochastic retarded systems with
  {M}arkovian switching.
\newblock {\em IEEE Transactions on Automatic Control}, 54(8):1898--1902,
  August 2009.

\bibitem[JP09]{julius1}
A.~A. Julius and G.~J. Pappas.
\newblock Approximations of stochastic hybrid systems.
\newblock {\em IEEE Transaction on Automatic Control}, 54(6):1193--1203, June
  2009.

\bibitem[JZst]{Jagtap2017}
P.~Jagtap and M.~Zamani.
\newblock {QUEST}: A tool for state-space quantization-free synthesis of
  symbolic controllers.
\newblock In {\em International Conference on Quantitative Evaluation of
  Systems}, pages 309--313. Springer International Publishing, September 2017,
  \url{https://www.hcs.ei.tum.de/en/software/quest/}.

\bibitem[KM13]{6580622}
I.~Kolmanovsky and T.~Maizenberg.
\newblock Stochastic optimal control of jump diffusion excited energy
  harvesters.
\newblock In {\em 2013 American Control Conference}, pages 5049--5055, June
  2013.

\bibitem[Kus67]{kushnerbook}
H.~J. Kushner.
\newblock {\em Stochastic Stability and Control}.
\newblock New York: Academic Press, 1967.

\bibitem[LSZ19]{Abolfazl2019}
A.~Lavaei, S.~Soudjani, and M.~Zamani.
\newblock Compositional construction of infinite abstractions for networks of
  stochastic control systems.
\newblock {\em Automatica}, 107:125--137, 2019.

\bibitem[LTOM12]{6425981}
J.~{Liu}, U.~{Topcu}, N.~{Ozay}, and R.~M. {Murray}.
\newblock Reactive controllers for differentially flat systems with temporal
  logic constraints.
\newblock In {\em 2012 IEEE 51st IEEE Conference on Decision and Control
  (CDC)}, pages 7664--7670, Dec 2012.

\bibitem[MPS95]{maler1995synthesis}
O.~Maler, A.~Pnueli, and J.~Sifakis.
\newblock On the synthesis of discrete controllers for timed systems.
\newblock In {\em Annual Symposium on Theoretical Aspects of Computer Science},
  pages 229--242. Springer, 1995.

\bibitem[{\O}S05]{Oks_jump}
B.~{\O}ksendal and A.~Sulem.
\newblock {\em Applied Stochastic Control of Jump Diffusions}.
\newblock Universitext. Springer-Verlag, Berlin, 2005.

\bibitem[PGT08]{pola2008approximately}
G.~Pola, A.~Girard, and P.~Tabuada.
\newblock Approximately bisimilar symbolic models for nonlinear control
  systems.
\newblock {\em Automatica}, 44(10):2508--2516, October 2008.

\bibitem[PPDB15]{pola}
G.~Pola, P.~Pepe, and M.~D. Di~Benedetto.
\newblock Symbolic models for time-varying time-delay systems via alternating
  approximate bisimulation.
\newblock {\em International Journal of Robust and Nonlinear Control},
  25(14):2328--2347, 2015.

\bibitem[PPDBT10]{pola2010symbolic}
G.~Pola, P.~Pepe, M.~D. Di~Benedetto, and P.~Tabuada.
\newblock Symbolic models for nonlinear time-delay systems using approximate
  bisimulations.
\newblock {\em Systems \& Control Letters}, 59(6):365--373, June 2010.

\bibitem[Sha13]{shaikhet2013lyapunov}
L.~Shaikhet.
\newblock {\em Lyapunov functionals and stability of stochastic functional
  differential equations}.
\newblock Springer Science \& Business Media, 2013.

\bibitem[Sko09]{skorokhod2009asymptotic}
A.~Skorokhod.
\newblock {\em Asymptotic methods in the theory of stochastic differential
  equations}.
\newblock American Mathematical Soc., 2009.

\bibitem[Tab09]{tabuada2009verification}
P.~Tabuada.
\newblock {\em Verification and control of hybrid systems: a symbolic
  approach}.
\newblock Springer Science \& Business Media, 2009.

\bibitem[ZA14]{zamani2014approximately}
M.~Zamani and A.~Abate.
\newblock Approximately bisimilar symbolic models for randomly switched
  stochastic systems.
\newblock {\em Systems \& Control Letters}, 69:38--46, July 2014.

\bibitem[ZAG15]{zamani2015symbolic}
M.~Zamani, A.~Abate, and A.~Girard.
\newblock Symbolic models for stochastic switched systems: {A} discretization
  and a discretization-free approach.
\newblock {\em Automatica}, 55:183--196, May 2015.

\bibitem[ZMEM{\etalchar{+}}14]{zamani2014symbolic}
M.~Zamani, P.~M.~Esfahani, R.~Majumdar, A.~Abate, and J.~Lygeros.
\newblock Symbolic control of stochastic systems via approximately bisimilar
  finite abstractions.
\newblock {\em IEEE Transactions on Automatic Control}, 59(12):3135--3150,
  December 2014.

\bibitem[ZTA17]{zamani2016towards}
M.~Zamani, I.~Tkachev, and A.~Abate.
\newblock Towards scalable synthesis of stochastic control systems.
\newblock {\em Discrete Event Dynamic Systems}, 27(2):341--369, June 2017.

\end{thebibliography}

\end{document}